\documentclass{article}

\usepackage[english]{babel}
\usepackage[utf8]{inputenc}
\usepackage[T1]{fontenc}
\usepackage{lmodern}
\usepackage{graphicx}
\usepackage{amsfonts}
\usepackage[nice]{nicefrac}
\usepackage{amsmath,amsthm,amssymb} %,mathabx}
\usepackage{thmtools}
\usepackage{caption}
\usepackage{array}
\usepackage{color}
\usepackage{url}
\usepackage[top=2cm,bottom=2.5cm,right=2cm,left=2cm]{geometry}
\usepackage{enumerate}
\usepackage{hyperref}

\usepackage{subfig}
\newenvironment{rmq}[0]{\vspace{3mm}\noindent \textbf{Remark: }}{\vspace{3mm}}

\newcommand{\bs}{\boldsymbol}
\newcommand{\PP}{\mathtt{P}}
\newcommand{\N}{\mathbb{N}}
\font\dsrom=dsrom10 scaled 1200
\def \indi{\textrm{\dsrom{1}}}

\newcommand{\st}{\, |\,}
\newcommand{\cst}{\mathtt{cst}}
\newcommand{\Var}{\mathtt{Var}}

\renewcommand{\P}{\mathbb{P}}
\newcommand{\ssup}[1] {{\scriptscriptstyle{({#1}})}}

\newcommand{\heap}[2]  {\genfrac{}{}{0pt}{}{#1}{#2}}
\newcommand{\sfrac}[2] {\mbox{$\frac{#1}{#2}$}}

\declaretheorem[thmbox=S, numberwithin=section, name=Theorem]{thm}
\newtheorem{prop}[thm]{Proposition}

\newtheorem{lem}[thm]{Lemma}

\makeatletter 
\renewenvironment{proof}[1][Proof] {\par\pushQED{\qed}\normalfont\topsep6\p@\@plus6\p@\relax\trivlist\item[\hskip\labelsep\bfseries#1\@addpunct{.}]\ignorespaces}{\popQED\endtrivlist\@endpefalse} 
\makeatother

\title{Condensation and symmetry-breaking in the zero-range process\\ with weak site disorder}
\author{C{\'e}cile Mailler\thanks{Department of Mathematical Sciences, University of Bath, Claverton Down, BA2 7AY Bath, UK. c.mailler/maspm@bath.ac.uk},
Peter M{\"o}rters\footnotemark[1]~~and Daniel Ueltschi\thanks{Department of Mathematics, University of Warwick, Coventry CV4 7AL, UL. daniel@ueltschi.org}}
\date{\scriptsize{\today}}

\begin{document}

\thispagestyle{empty}
\maketitle

\vspace{-2\baselineskip}

\begin{abstract}
\noindent 
Condensation phenomena in particle systems typically occur as one of two distinct types: either as a \emph{spontaneous} symmetry breaking in a homogeneous system, in which particle 
interactions enforce condensation in a randomly located site, or as an \emph{explicit} symmetry breaking in a system with background disorder, in which particles
condensate in the site of extremal disorder. In this paper we confirm a recent conjecture by Godr\`eche and Luck by showing, for a zero range process with weak site disorder, 
that there exists a phase where condensation occurs with an intermediate type of symmetry-breaking, in which particles condensate in a site randomly chosen from a range of sites 
favoured by disorder. We show that this type of condensation is characterised by the occurrence of a Gamma distribution in the law of the disorder at the condensation site. We further investigate fluctuations of the condensate size and confirm a phase diagram, again conjectured by Godr\`eche and Luck, showing the existence of phases with normal and anomalous~fluctuations. 
\end{abstract}

\small\tableofcontents\normalsize

\section{Motivation and background}

The purpose of this paper is two-fold. The \emph{first} purpose is to show that for certain low-dimensional particle systems far from equilibrium the simultaneous presence of inter-particle 
interactions and interactions of particles with a spatial disorder can lead to a novel form of symmetry breaking,  occurring in a phase when the two competing particle forces are of 
comparable strength. In these systems we observe that, when the particle density exceeds a certain threshold value, the excess fraction of the particles condensates in a single site.
This site is neither chosen uniformly at random (as would be the case in systems with spontaneous symmetry breaking) nor as a function of the underlying site disorder (as would be the case
in systems with explicit symmetry breaking) but by a nontrivial random mechanism favouring sites with more extreme site disorder. The existence of such systems was predicted in a recent paper by Godr\`eche and Luck~\cite{GL12}.  The \emph{second} purpose of this paper is to give a further example of the ubiquity of the Gamma distribution in particle systems with condensation, which was first observed in  Dereich and M\"orters~\cite{DM13}. In our context the Gamma distribution occurs as the universal distribution of the disorder at the condensation site. 
\medskip

\pagebreak[3]

The interacting particle model under consideration here is the  \emph{zero-range process}, first introduced in the mathematcial literature by Spitzer in~\cite{Spitzer}. 
The zero-range process has gained importance in the statistical mechanics literature, for example as a  generic model for domain wall dynamics in a system far from equilibrium~\cite{KLMST02} or as a model for granular flow~\cite{E00, CG10}. It is also a particularly simple model undergoing a condensation transition, and widely studied for this reason alone~\cite{GSS03,  EFGM05, AGL13}.
It is related to the ideal Bose gas and to spatial permutations~\cite{EJU14}.
The zero-range process has also been studied in a disordered medium, both in infinite~\cite{AFGL00} and finite~\cite{EH05} geometries, and the latter situation is also the context of the present paper.
\medskip

Our version of the zero-range process is a continuous time Markov process, which  can be described as a system of  $m$~indistinguishable particles each located 
in one of $n$~different sites. Every site can hold an arbitrary number of particles. At each time instance particles move independently given the 
particle configuration, and the rate at which particles hop from position $i$ to a different position $j$ is given as $q_{ij} u_k$, where $k$ is the 
number of particles at site~$i$. Here $(q_{ij}\colon 1\leq i,j\leq n)$ is a $Q$-matrix (i.e.\ off-diagonal entries are nonnegative and each row sums 
to zero) describing the unconstrained particle motion, and $(u_k \colon k\geq 0)$ is a sequence of nonnegative weights with $u_0=0$, that describes
the particle interactions. The term zero-range process comes from the fact that, at any given time instance, the interaction is only between particles 
in the same site or, in other words, the jump rate above depends on the global particle configuration only through the number~$k$ of particles on the 
site of departure. The case $u_k=k$ corresponds to independent movement of the particles without interaction, but our interest here is mainly in 
sublinear sequences  $(u_k \colon k\geq 0)$, in which particles move slower if they are aggregated at a site with many other particles. One such case 
would be that $u_k=1$, for all $k>0$, meaning that at every site only one particle is free to move. The phenomena of interest in this paper occur when
$u_k$ is given as a small perturbation of this  case.
\medskip

Assuming that the finite state Markov chain described above is irreducible, general theory insures that the state of the zero-range process 
converges in law, as time goes to infinity, to a unique stationary distribution, or steady state. Denoting by $Q_i$ the number of particles 
located in site~$i$ this  distribution is explicitly given by
$${P}\big(Q_1=q_1,\ldots,Q_n=q_n\big) = \frac1{Z_{m,n}}\, \prod_{i=1}^n \pi_i^{q_i} p_{q_i} \quad \mbox{ if } q_i\geq0 \mbox{ are integers with }\sum_{i=1}^n q_i=m,$$   
where $(\pi_i \colon 1\leq i\leq n)$ is a positive left eigenvector of $Q$ for the eigenvalue zero,  
$(p_k \colon k\geq 0)$ are derived from $(u_k \colon k\geq 0)$ by $p_0=1$ and
$p_k=1/u_1\cdots u_k$, for $k\geq 1$, and $Z_{m,n}$ is the normalisation constant, or partition function. The most studied case is that of 
spatial homogeneity in which $(\pi_i \colon 1\leq i\leq n)$ is a vector of constant (nonzero) entries. Already in this simple case the phenomenon 
of condensation can occur, as established in the seminal paper by Gro\ss kinsky et al.~\cite{GSS03}.
%\medskip
%
In the set-up above, the particle system above allows for general spatial inhomogeneities encoded in the \mbox{$Q$-matrix}. Following 
Godr\`eche and Luck~\cite{GL12} in this point, we now simplify the analysis by focusing on relatively simple spatial inhomogeneities, which 
are chosen to display the full richness of possible behaviour. To this end we replace the invariant measure of a single particle motion 
$(\pi_i \colon 1\leq i\leq n)$  by a random environment given as a product of a random site disorder. More precisely, we are assuming that 
$\pi_i=X_i$, for $1\leq i\leq n$, where $(X_i \colon i\in\mathbb{N})$ is a sequence of independent, identically 
distributed random variables. We think of $X_i$ as the fitness of site~$i$, where fitter sites are a more attractive host for particles. 
One of many possible dynamics that give rise to this stationary behaviour is if sites are arranged as a circle, and particles located at site~$i$
with occupancy~$k$ hop clockwise  to their neareast neighbour with rate $u_k/X_i$. As our results  can be expressed in terms of the stationary distribution 
without explicit reference to any particle dynamics, we do not have to make explicit reference 
to the particle dynamics or the $Q$-matrix underlying our random environment. While this approach enables a rigorous mathematical analysis of the key phenomena, its downside is that our results contain no direct information about the kinetics of the zero-range process. 
\medskip

Our results on this model take the form of limit results where $n$, the number of sites, and $m$, the number of particles, 
go to infinity so that the ratio $m/n$ converges to a fixed density~$\rho>0$. 
We assume that the random variable~$X$ determining the site fitness is bounded from above, 
without loss of generality by the value 1,  and that its distribution function is regularly varying at 1 with index~$\gamma$, for some $\gamma>0$.
The sequence $(p_k \colon k\geq 0)$ is assumed to be regularly varying with index $-\beta$, for some $\beta>1$.
The phase diagrams we identify in our main results will be given in terms of the parameters $\beta$ and $\gamma$. 
\medskip

We first show in Theorem~\ref{th:condensation} that if $\beta+\gamma>2$, there exists a positive and finite critical density
$\rho^{\star}$ such that if $\rho>\rho^{\star}$, with probability going to one, there exists a unique site carrying a positive fraction of the particles. This fraction converges to $\rho-\rho^\star>0$. This is the phenomenon of \emph{condensation}. 
\medskip

\pagebreak[3]

If condensation occurs, we ask 
\begin{itemize}
\item[(1)] At which site does the condensation occur?
\item[(2)] What is the fitness of the site at which condensation occurs?
\item[(3)] How does the condensate fraction fluctuate around the limit $\rho-\rho^\star$?
\end{itemize}
Our main results answer these three questions.
In Theorem~\ref{th:I_n} we address the first question. We show that in the case $\gamma>1$,
condensation occurs at the site with highest fitness value,  revealing a case of \emph{explicit} symmetry breaking. 
If $\gamma\leq 1$ and $\beta+\gamma>2$ however, with high probability, condensation occurs at a site chosen 
from a range of sites with high fitness. We describe the non-degenerate
limiting  distribution for the rank order of the condensation site. This result establishes the novel phenomenon of 
\emph{intermediate} symmetry breaking conjectured by Godr\`eche and Luck~\cite{GL12}. 
%\smallskip
%
The second question is addressed in Theorem~\ref{th:gamma}, where we show that in the phase of intermediate symmetry breaking
the fitness of the condensation site satisfies a universal limit theorem. In fact, regardless of the underlying fitness distribution, the
disorder of the condensation site converges, appropriately scaled, 
to a Gamma distribution. Recall that the Gamma distribution is not a classical extreme value distribution, so that its occurence in this context may be considered surprising.
%that, in contrast to the above,  the scaled difference of the essental supremum and the fitness at the site of maximal fitness converges to %an exponential distribution.
%\smallskip
%
In Theorem~\ref{thm:fluctuations} we address the third question by studying the quenched fluctuations in the size of the condensate
in the case $\gamma\leq 1$ of weak disorder. We show that, if $\beta+\gamma\geq 3$, the fluctuations around a disorder dependent  finite size approximation of the limiting value $\rho-\rho^{\star}$ are normal. In contrast to this, if $2<\beta+\gamma<3$, the fluctuations are stable with index $\beta+\gamma-1$. In the (easier) annealed setup such a behaviour was also conjectured by  Godr\`eche and Luck~\cite{GL12}. %
\medskip%

Our proofs are mainly based on a careful analysis of  a \emph{grand-canonical ensemble}, a sequence of independent but not identically distributed random variables $Q_1, Q_2,\dots$ with the law of $Q_i$ given explicitly in terms of the fitness~$X_i$. Conditioning on the event $Q_1+\cdots+Q_n=m$ we obtain the distribution of site occupancies in the stationary zero range model with $m$ particles and $n$ sites, often referred to as the \emph{canonical ensemble}. Although the behaviour of the ensembles is radically different in the case of condensation, the key idea is still to derive properties of the canonical ensemble from much more accessible properties of the grand-canonical ensemble. For example, 
we show that the number of particles outside the condensation site in the canonical ensemble is well-approximated by the sum $Q_1+\cdots+Q_n$ of independent random variables in the grand-canonical ensemble. The latter quantity can then be studied by 
classical means. This technique is inspired by ideas of Janson~\cite{Janson12} for a model without disorder. Adaptation of these
ideas to the study of disordered systems is the main technical innovation of this paper.
%\bigskip
 
\vspace{\baselineskip}
{\bf Notation: }
The symbol $\cst$ stands for a positive constant which may change its value at every apperance.
Given two sequences $(u_n)_{n\geq 1}$ and $(v_n)_{n\geq 1}$, we write $u_n\sim v_n$ if 
\smash{$\nicefrac{u_n}{v_n} \to 1$}. We write $u_n = o(v_n)$, or $u_n \ll v_n$,  if  
\smash{$\nicefrac{u_n}{v_n} \to 0$}.
We use the symbol $u_n={O}(v_n)$ if there exists $c> 0$ such that $|u_n|\leq c\, |v_n|$ for all sufficiently large~$n$, and 
indicate by $O_\P$ if the implied constant $c$ is allowed to be a random variable under~$\P$. We write 
$u_n = \Theta(v_n)$ if both $u_n={O}(v_n)$ and $v_n={O}(u_n)$ hold. Finally, given a sequence $\delta_n\to 0$ and a function $f$, we write 
$u_n \approx f(v_n \pm \delta_n)$ if $f(v_n - \delta_n) \leq u_n \leq f( v_n +\delta_n)$ 
for all sufficiently large~$n$.

\vspace{\baselineskip}
{\bf Acknowledgements: } The authors are supported by EPSRC through the project EP/K016075/1.
We are grateful to Martin Hairer, Roman Koteck\'y, Victor Rivero, Vitali Wachtel, and Matthias Winkel 
for fruitful discussions on various aspects of this paper.

\section{Statement of the main results}

Let $\mu$ be a probability distribution on $[0,1]$ satisfying, for some $\gamma>0$,
\begin{equation}\label{eq:rvmu}\tag{$\mathtt{RV\mu}$}
\mu([1-x, 1])\sim \alpha_1 \, x^{\gamma}, \text{ when } x\downarrow 0,
\end{equation}
%For practical reasons, we will also assume that $\mu$ admits a density $m$ on $[0,1]$, which thus verifies
%$m(1-x) \sim x^{\gamma}$ when $x\to 0$.
and $(p_k)_{k\geq 0}$ a probability distribution on $\mathbb{N}_0:=\{0,1,2,\ldots\}$ 
such that, for some $\beta>1$,
\begin{equation}\label{eq:rvp}\tag{$\mathtt{RVp}$}
p_k \sim \alpha_2\, k^{-\beta}, \quad\text{ as } k \uparrow\infty.
\end{equation}
 We believe that all our results, except the fluctuation result at the end of this section, hold \emph{mutatis mutandis} if the positive constants $\alpha_1$, $\alpha_2$ were replaced by slowly varying functions. This would require a greater technical effort, which would not help the understanding of the phenomena we are interested in, 
and would be detrimental to the readability of the proofs. 
\pagebreak[3]

%\medskip

We always assume, without loss of generality, that $p_0>0$. Denote by $\Phi\colon[0,1]\to[0,1]$ the generating function of the distribution $(p_k)_{k\geq 0}$, given by 
$$\Phi(z) = \sum_{k=0}^\infty p_k z^k,$$
and % define, for a random variable~$X$ with distribution~$\mu$, the critical density
define the critical density
\[\rho^{\star}:= \int \frac{x \, \Phi'(x)}{\Phi(x)} \, \mu(dx).\]
%\mathtt E \Big[ \frac{X\Phi'(X)}{\Phi(X)} \Big].\]
The random disorder in our model is given by and i.i.d.\ sequence
$\bs X=(X_i \colon i\in\N)$ of random variables with distribution~$\mu$. Given the 
disorder, the stationary distribution of the disordered zero-range process is given by
\begin{equation}\label{eq:stat_dist}
P_{\bs X} (Q_1 = q_1, \ldots, Q_n = q_n) = \frac1{Z_{m,n}}\ \prod_{i=1}^n X_i^{q_i} p_{q_i} \qquad \mbox{ for } 
q_1,\ldots,q_n\in\N_0 \mbox{ with } q_1 + \cdots + q_n = m,
\end{equation}
where  $Z_{m,n}$ is the normalisation constant. We write $P_{\bs X}$ for the `quenched' law of $(Q_1,\dots,Q_n)$ given $\bs X$,
we write $\mathtt P$ for the law of the disorder~$\bs X$, and ${P}_{m,n}=\mathtt E P_{\bs X}$ for the `annealed' law, the joint law of  $(X_1,\ldots, X_n)$ and $(Q_1,\dots,Q_n)$ with $Q_1 + \cdots + Q_n = m$ and $\rho_n:=\nicefrac{m}{n}\to\rho>0$. 
\bigskip

Denote by $(Q^{(1)}_n, \ldots, Q^{(n)}_n)$ the order statistics of $(Q_1, \ldots, Q_n)$. Our first result shows that in the condensation
regime $\beta+\gamma>2$, if the particle density~$\rho$ exceeds the critical value $\rho^{\star}$, the excess particles form a 
condensate of macroscopic occupancy in exactly one site.
\bigskip

\begin{thm}[Condensation]\label{th:condensation}
Suppose $\beta+\gamma>2$. Then $\rho^\star<\infty$ and if $\rho > \rho^{\star}$  then, with high ${P}_{m,n}$-probability,
\[Q_n^{(1)} = (\rho-\rho^{\star}) n + o(n) 
\quad \text{ and }\quad
Q_n^{(2)} = o(n).\]
\end{thm}
\bigskip

The following two theorems show that in the case $\gamma< 1$ the condensate does not normally sit in the site with the largest fitness. 
This is called the `extended condensate case' by Godr\`eche and Luck, but we prefer the term 
intermediate symmetry-breaking to emphasise that the condensate is still located at a single site and not extended over several sites.
We say that a sequence of random variables $(Z_n)_{n\in\N}$   {\bf converges in quenched distribution}  to the random variable $Z$ if, 
for all $\varepsilon >0$ and all $u\in \mathbb R$,
\begin{equation}\label{iqddef}
\mathtt P \left(\left|P_{\bs X}(Z_n \leq u)-P_{\bs X}(Z \leq u)
\right|>\varepsilon\right)
\to 0,\quad \text{ when }n\uparrow\infty.
\end{equation}
We denote by $I_n$ the index of the site of maximal occupancy, so that $Q_{I_n} = Q_n^{(1)}$. By Theorem~\ref{th:condensation}
this eventually defines $I_n$ uniquely in the condensation regime. We further let
$K_n$ be the rank order of the fitness of the condensation site, i.e. $K_n=k$ if and only if
$$\big| \big\{i\in\{1,\dots,n\} \colon X_i > X_{I_n} \big\}\big| = k-1.$$
Recall that the density of a Gamma distributed random variable 
with parameters $(\gamma,\lambda)$ is given by
\[p(x) = \sfrac{\lambda^{\gamma}}{\Gamma(\gamma)} \, \,
x^{\gamma-1} \mathtt{e}^{-\lambda x}
\qquad \mbox{ for }  x\geq 0.\]
\medskip

\begin{thm}[Fitness rank of the condensate]\label{th:I_n} 
\begin{enumerate}[(i)]
\item If $\gamma>1$ and $\rho>\rho^{\star}$, then with high ${P}_{m,n}$-probability we have $K_n=1$.
\item If $\gamma< 1$, $\beta+\gamma>2$ and $\rho>\rho^{\star}$, then
$$\left(n^{\gamma-1}K_n\right)^{\nicefrac{1}{\gamma}} \to K$$ 
in quenched distribution, where $K$ is a Gamma distributed random variable 
of parameters $(\gamma, \frac{\rho-\rho^{\star}}{\alpha_1^{_{\nicefrac1\gamma}}})$.
%(\rho-\rho^{\star})\alpha_1^{-\nicefrac1\gamma})$.
\end{enumerate}
\end{thm}
\bigskip

Note that the two phases described in Theorem~\ref{th:I_n} are both condensation phases, in
case~(i) explicit symmetry breaking occurs, while in case~(ii) there is intermediate symmetry 
breaking. Figure~1 illustrates the phase diagram established in Theorem~\ref{th:I_n}.
\pagebreak[3]

\begin{figure}[h]
\begin{center}
\includegraphics[width=.3\textwidth]{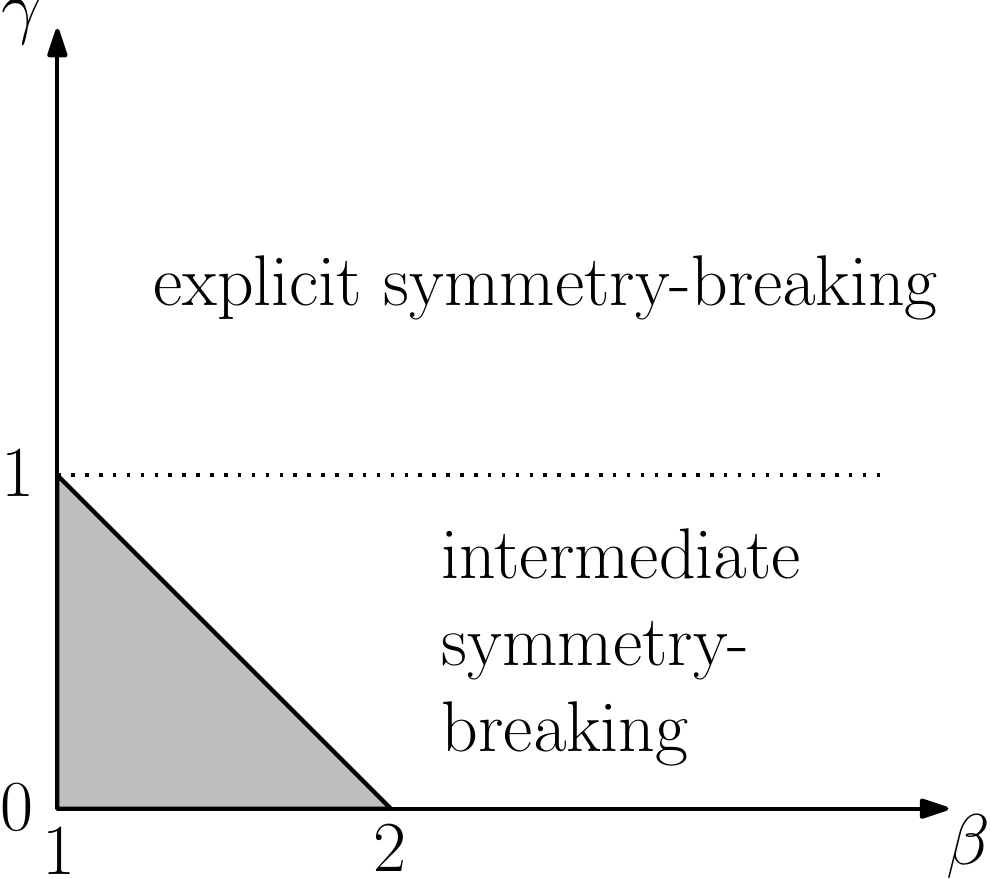}
\end{center}
\caption{This phase diagram shows the behaviour of the disordered zero-range process according to its two parameters $\beta>1$ and $\gamma>0$. The grey part is a zone where there is no condensation, there is condensation in the white part as soon as $\rho>\rho^{\star}$.}
\label{fig:phases}
\end{figure}

The next theorem gives the universal law of the  fitness of the condensate.
\bigskip

\begin{thm}[Fitness of the condensate]\label{th:gamma}
If $\gamma<1$, $\beta+\gamma>2$ and $\rho>\rho^{\star}$,  denote by $F_n = X_{I_n}$ the fitness at the condensation site.
Then $$n(1-F_n) \to F$$ in quenched distribution, where~$F$ is a Gamma distributed random variable
with parameters $(\gamma, \rho-\rho^{\star})$.
\end{thm}

Finally,   we have very precise results about the asymptotic behaviour of 
the size of the condensate in the case of intermediate symmetry-breaking. 
We define random variables
\[\nu_n:=\frac1n \sum_{i=1}^n \frac{X_i \Phi'(X_i)} {\Phi(X_i)},\]
and note that $\mathtt E \nu_n=\rho^\star$. 
%\sum_{k=1}^\infty k p_k X_i^k.\]
We shall see that the  first order estimate of $Q_n^{(1)}$ given the disorder is $m-\nu_n n$,
which divided by~$n$ converges in $\mathtt P$-probability to~$\rho-\rho^\star$.  The following theorem describes the fluctuations  of \smash{$Q_n^{_{(1)}}$} 
around the value~$m-\nu_n n$.
\medskip

\begin{thm}[Quenched fluctuations of the condensate]\label{thm:fluctuations}
Assume that $\gamma< 1$.
\begin{enumerate}[(i)]
\item If $2<\beta+\gamma<3$ and $\rho>\rho^{\star}$, let $\kappa = \frac{1}{\beta+\gamma-1}$. Then,
\[\frac{Q_n^{(1)}-m+\nu_n n}{n^{\kappa}} \to W_{\kappa}\]
in quenched distribution, where $W_{\kappa}$ is a $\nicefrac1\kappa$-stable random variable.

\item If $\beta+\gamma\geq 3$ and $\rho>\rho^{\star}$, then
\[\frac{Q_n^{(1)}-m+\nu_n n}{\sqrt n} \to W\]
in quenched distribution, where $W$ is a
normal random variable.
\end{enumerate}
\end{thm}

\begin{rmq}
Note that the \emph{quenched} fluctuation result gives information on the size of the condensate for fixed instances of the disorder and is much more subtle than the  \emph{annealed} fluctuation results that would allow averaging over the disorder. Annealed fluctuations are centred around $(\rho-\rho^\star)n$ and hold without the restriction $\gamma< 1$, the distinction of the normal and anomalous regime persists in this situation.
\end{rmq}
\pagebreak[3]

\begin{figure}[ht]
\begin{center}
\includegraphics[width=.3\textwidth]{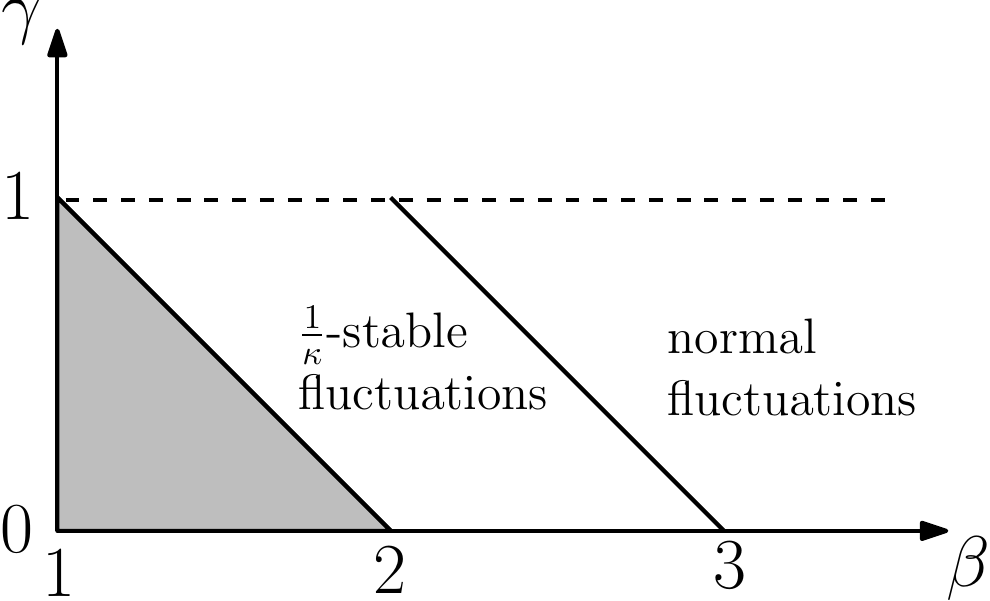}
\end{center}
\caption{This phase diagram shows the fluctuations of the size of the condensate according to the values of the two parameters $\beta$ and $\gamma$.}
\label{fig:fluctuations}
\end{figure}

\vspace{\baselineskip}

\pagebreak[3]

The following four sections are devoted to the proofs of our main theorems. Section~\ref{sec:LLN} presents the grand canonical framework used in our proofs. It contains a fairly standard technical proofs of a central limit theorem for independent random variables that may be skipped on first reading.
Section~\ref{sec:condensation} contains the proof of condensation, i.e. of Theorem~\ref{th:condensation}.
Section~\ref{sec:gammas} is devoted to intermediate symmetry-breaking and contains the proofs of Theorems~\ref{th:I_n} and~\ref{th:gamma}. Section~\ref{sec:fluctuations} deals with fluctuations,  this is where Theorem~\ref{thm:fluctuations} is proved.
We list some interesting open problems in Section~7, and in Appendix~\ref{app:X} we collect general results on the limit behaviour 
of the fitnesses, which are used throughout the paper.  As results on i.i.d. random variables regularly varying near their essential 
supremum are difficult to find in the literature, this may be of independent interest.
\bigskip

\pagebreak[3]

\section{The grand canonical ensemble}\label{sec:LLN}

Given the sequence  $X_1, X_2, \ldots$  of random variables with distribution~$\mu$ we now define
another model,  the \emph{grand canonical ensemble}, as the sequence $Q_1, Q_2, \ldots$ of conditionally 
independent random variables with the law of $Q_i$ given by
\begin{equation}
\label{gd-can law}
\mathbb{P}_{\bs X}(Q_i=k) = \frac{p_k X_i^k}{\Phi(X_i)}.
\end{equation}
%Remark also that, without loss of generality, we can assume $p_0>0$: if $p_0 = 0$, one can consider %instead of the $(Q_i)_{i\geq 1}$ 
%the shifted random variables $\tilde Q_i = Q_i-1$.
Given positive integers $n,m$ we can recover $P_{m,n}$ as the law of 
$(Q_1, \ldots, Q_n, X_1, \ldots, X_n)$ conditioned on the event $\{Q_1+\cdots +Q_n = m\}$.
In this framework the random variables $\nu_n$ can be described as
\[\nu_n =\frac1n \sum_{i=1}^n \mathbb E_{\bs X} Q_i.\]
We now show that the sequence $(\nu_n)_{n\in\N}$ satisfies a law of large numbers.

\begin{lem}[Natural density]\label{lem:LLN_nu}
If $\beta+\gamma>2$, then ${\mathtt P}$-almost surely $\nu_n \to \rho^{\star}<\infty$.
\end{lem}

\begin{proof} % [Proof of Lemma~\ref{lem:LLN_nu}]
Denote \smash{$G(x):= \frac{x\Phi'(x)}{\Phi(x)}$}, for all $x\in [0,1]$. We show that $G(X)$ is integrable, so that the result follows from an application
of Kolmogorov's law of large numbers. In the case $\beta > 2$, we have that $G$ is bounded and hence integrable. In the case $2-\gamma<\beta< 2$, we get 
$\Phi'(x)= \Theta\left((1-x)^{\beta-2}\right)$ (as a consequence of~\cite[Theorem~VI.3]{FS09}), 
and hence $G(x) = \Theta((1-x)^{\beta-2})$, as $x\uparrow 1$. 
%, since $0<p_0<\Phi(x)\leq 1$, for all $x\in [0,1]$. 
Letting $G^{-1}(u)=\inf\{x \colon G(x)>u\}$, we have
$\mathtt P(G(X)>u) \leq \mathtt P(X  \geq  G^{-1}(u))$.
Observe that $G^{-1}(u)\uparrow 1$,  as $u\uparrow\infty$, which tells us in view of~\eqref{eq:rvmu} that
$ \mathtt P(X  \geq  G^{-1}(u)) \sim \alpha_1 (1-G^{-1}(u))^{\gamma},$ as~$u\uparrow\infty$.
As \smash{$1-G^{-1}(u)= \Theta(u^{\frac{1}{\beta-2}})$}, we obtain
\smash{$\mathtt P(G(X)>u)={O}(u^{-\frac{\gamma}{2-\beta}})$}.
Integrability follows as 
\smash{$\frac{\gamma}{2-\beta} > 1$}.  Finally, in the case $\beta = 2$, we have $\Phi'(x)\sim -\log (1-x)$ and 
integrability follows using a similar argument as above.
\end{proof}

Limit theorems for the independent (but not identically distributed) random variables
\smash{$(Q_i)_{i\geq 1}$ under $\mathbb P_{\bs X}$} are nontrivial, but can be obtained by
classical methods. We abbreviate the partial sums as 
$$S_n:=\sum_{i=1}^n Q_i.$$

\begin{lem}[Grand canonical law of large numbers]\label{lem:LLN_Q}
If $\beta+\gamma>2$, then $\frac1n\,S_n-\nu_n \to 0$
in $\mathbb P_{\bs X}$-probability.
\end{lem}

Combining Lemmas~\ref{lem:LLN_nu} and~\ref{lem:LLN_Q} we see that, if $\rho>\rho^\star$, the probability
$\mathbb P_{\bs X}(S_n= m)$ is going to zero as $n\to\infty$. We shall
see later\footnote{See, in particular, Lemma~\ref{lem:calcul_proba_E1}.} that, 
with high $\mathtt P$-probability, this decay is polynomial if $\gamma\leq 1$, 
but stretched exponential if $\gamma>1$.

\bigskip

The law of large numbers, Lemma~\ref{lem:LLN_Q}, follows from the central limit theorem for the grand canonical ensemble, which we now state. 
The central limit theorem for the grand canonical ensemble prepares the proof
of Theorem~\ref{thm:fluctuations} for the canonical ensemble.  The proof is a direct application of classical techniques for independent (but not identically distributed) random variables, and may be omitted on first reading.

\begin{prop}[Grand canonical central limit theorem]\label{prop:TCL}
\ \\[-5mm]
\begin{enumerate}[(i)]
\item If $2< \beta+\gamma < 3$,  let $\kappa = \frac{1}{\beta+\gamma-1}$. Then, in quenched distribution\footnote{To define convergence in quenched distribution in the grand-canonical framework, one has to replace $P_{\bs X}$ by $\mathbb P_{\bs X}$ in~\eqref{iqddef}.},
\[\frac{\sum_{i=1}^n Q_i - \nu_n n}{n^{\kappa}} \to W_\kappa,\]
where $W_\kappa$ is a $\nicefrac1\kappa$-stable random variable.
\item If $\beta+\gamma \geq 3$, %and $\gamma<2$, 
then, in quenched distribution,
\[\frac{\sum_{i=1}^n Q_i - \nu_n n}{\sqrt n} \to W,\]
where $W$ is a Gaussian random variable.
\end{enumerate}
\end{prop}

\begin{proof}[Proof of Proposition~\ref{prop:TCL}]
$(ii)$ This is a direct application of the central limit theorem for sums of independent but non identical random variables based on Lindeberg's condition, i.e.,  for all $\varepsilon>0$,
\[\lim_{n\uparrow\infty} \frac 1 n \sum_{i=1}^n 
\mathbb E \left[(Q_i-\mathbb E_{\bs X} Q_i)^2 \indi\{|Q_i-\mathbb E_{\bs X} Q_i| > \varepsilon\sqrt n\}\right]
= 0.\]
Recall that \smash{$\mathbb E_{\bs X} Q_i = \frac{X_i\Phi'(X_i)}{\Phi(X_i)} = G(X_i)$}
and that, if $\beta>2$, the function \smash{$G(x)= \frac{x\Phi'(x)}{\Phi(x)}$} 
is bounded on $[0,1]$, behaves as 
$O((1-x)^{\beta-2})$ if $\beta <2$ and as $O(-\log (1-x))$ if $\beta = 2$.
Therefore, in view of Lemma~\ref{lem:extremal_properties_X} and using that $\beta+\gamma\geq 3$ and $\beta >1$, we have
$\max_{i=1}^n G(X_i) = o(\sqrt n)$ in $\mathtt P$-probability.
Therefore, for all large enough $n$,
\[
\sum_{i=1}^n \mathbb E \left[(Q_i-\mathbb E_{\bs X} Q_i)^2 \indi\{|Q_i-\mathbb E_{\bs X} Q_i| > \varepsilon\sqrt n\}\right]
= \sum_{i=1}^n \sum_{k= \mathbb E_{\bs X} Q_i + \varepsilon\sqrt n}^{\infty}
\frac{p_k X_i^k}{\Phi(X_i)} (k-\mathbb E_{\bs X}Q_i)^2
\leq \cst. \sum_{k=\varepsilon \sqrt n}^{\infty}
k^{2-\beta}  \sum_{i=1}^n X_i^k.
\]
First note that assuming $\beta >3$ leads to
\[\frac1n \sum_{k=\varepsilon \sqrt n}^{\infty}
k^{2-\beta}  \sum_{i=1}^n X_i^k
\leq \sum_{k=\varepsilon \sqrt n}^{\infty}
k^{2-\beta} \to 0,\]
and hence Lindeberg's condition is verified. We may assume now that $\beta \leq 3$ and write
\[
\frac1n \sum_{k=\varepsilon \sqrt n}^{\infty}
k^{2-\beta}  \sum_{i=1}^n X_i^k
= \frac1n \sum_{k=\varepsilon \sqrt n}^{\frac{n^{\nicefrac1{\gamma}}}{\log n}}
k^{2-\beta}  \sum_{i=1}^n X_i^k
+ \frac1n \sum_{k=\frac{n^{\nicefrac1{\gamma}}}{\log n}}^{n^{\nicefrac1{\gamma}}\log^2 n}
k^{2-\beta}  \sum_{i=1}^n X_i^k
+ \frac1n \sum_{k=n^{\nicefrac1{\gamma}}\log^2 n}^{\infty}
k^{2-\beta}  \sum_{i=1}^n X_i^k,
\]
where the first and second term on the right are void if $\gamma>2$.
Applying Lemma~\ref{lem:approx_sum}$(ii)$ allows to bound the inner sum of the first term 
by a constant multiple of~$nk^{-\gamma}$, showing that the term tends to zero because $\beta+\gamma > 3$.
The second term is bounded from above by (we assume here that $\beta<3$, the case $\beta = 3$ 
can be treated similarly)
\[\frac1n \sum_{i=1}^n X_i^{\frac{n^{\nicefrac1{\gamma}}}{\log n}} \sum_{k=1}^{n^{\nicefrac1{\gamma}}\log^2 n} k^{2-\beta}
\sim \sfrac{\log^{\gamma} n}{n} \, n^{\frac{3-\beta}{\gamma}} \log^{2(3-\beta)} n,\]
using Lemma~\ref{lem:approx_sum}$(ii)$ applied to $s_n = \frac{n^{\nicefrac1{\gamma}}}{\log n}$.
Hence the second term also tends to zero as $n\uparrow\infty$.
Finally, the third term is, by Lemma~\ref{lem:approx_sum}$(i)$, and~\ref{lem:extremal_properties_X}, 
asymptotically bounded by
\begin{align*}
\frac1 n \, \sum_{k=n^{\nicefrac1{\gamma}}\log^2 n}^{\infty}
k^{2-\beta} \big(X_n^{(1)}\big)^k V_k^{(n)}
&\leq \cst.
\frac1 n \int_{\cst.n^{\nicefrac1{\gamma}}\log^2 n}^{\infty}
x^{2-\beta} \mathtt e^{-xn^{-\nicefrac1{\gamma}}} dx
\leq \cst.n^{\frac{3-\beta}{\gamma}-1} \int_{\log^2 n}^{\infty}
u^{2-\beta} \mathtt e^{-u} du,
\end{align*}
which also goes to zero, because $\beta+\gamma \geq 3$.
Therefore, Lindeberg's condition is verified concluding the proof of~$(ii)$.
Note that the variance of the limit normal distribution is given by $\mathtt E\  \Var_{\bs X} Q_i$.

\vspace{\baselineskip}
$(i)$ We apply the very general~\cite[\S 25, Theorem2]{GK}.
Using this it is enough to show that, asymptotically as $n\uparrow\infty$, there are
constants $C_1, C_2 \geq 0$ such that 
\begin{equation}\label{eq:cond_stable_th1}
\sum_{i=1}^n \mathbb P_{\bs X}\left(Q_i - \mathbb E_{\bs X} Q_i \geq x n^{\kappa}\right) \to \frac{C_1}{x^{\nicefrac1{\kappa}}}, \quad \text{ for all } x>0,
\end{equation}
\begin{equation}\label{eq:cond_stable_th2}
\sum_{i=1}^n \mathbb P_{\bs X}\left(Q_i - \mathbb E_{\bs X} Q_i \leq x n^{\kappa}\right) \to \frac{C_2}{|x|^{\nicefrac1{\kappa}}}, \quad \text{ for all }  x<0,
\end{equation}
\begin{equation}\label{eq:cond_stable_th3}
\lim_{\varepsilon\downarrow 0}\limsup_{n\uparrow\infty} \frac1{n^{2\kappa}}
\sum_{i=1}^n \Var_{\bs X}\big((Q_i-\mathbb E_{\bs X} Q_i) \indi\{|Q_i-\mathbb E_{\bs X} Q_i|<\varepsilon n^{\kappa}\}\big) = 0.
\end{equation}
First remark that, as above, we have  $\sup_{i=1}^n \mathbb E_{\bs X} Q_i = o(n^{\kappa})$. Hence, \eqref{eq:cond_stable_th2} is (trivially) verified with $C_2=0$.
Now recall that $p_k\sim \alpha_2 k^{-\beta}$ when $k$ tends to infinity.
Thus, for all $\varepsilon>0$ there exists an integer $k(\varepsilon)$ such that, for all $k\geq k(\varepsilon)$, we have
$p_k \approx (1\pm \varepsilon) \alpha_2 k^{-\beta}$.
Choose $n$ such that $xn^{\kappa} > k(\varepsilon)$ 
and such that $\sup_{i=1}^n \mathbb E_{\bs X} Q_i \leq x n^{\kappa}$. Then
\begin{align}
\sum_{i=1}^n \mathbb P_{\bs X}\left(Q_i - \mathbb E_{\bs X} Q_i \geq x n^{\kappa}\right)
&\approx \alpha_2 (1\pm \varepsilon) 
\sum_{i=1}^n \sum_{k\geq xn^{\kappa}+\mathbb E_{\bs X}Q_i} k^{-\beta} \frac{X_i^k}{\Phi(X_i)}\notag\\
&\approx \alpha_2 (1\pm \varepsilon) 
\sum_{i=1}^n \sum_{k\geq xn^{\kappa}} (k+\mathbb E_{\bs X}Q_i)^{-\beta} \frac{X_i^{k+\mathbb E_{\bs X}Q_i}}{\Phi(X_i)}.\label{eq:blabla}
\end{align}
To show that $(k+\mathbb E_{\bs X}Q_i)^{-\beta} \approx (1\pm \varepsilon) k^{-\beta}$ for all $k\geq xn^{\kappa}$, for all $i\in\{1, \ldots, n\}$, and large enough $n$, note that
\[k^{-\beta}\left(1+\frac{\sup_{i=1..n}\mathbb E_{\bs X}Q_i}{xn^{\kappa}}\right)^{-\beta}
\leq(k+\mathbb E_{\bs X}Q_i)^{-\beta} \leq k^{-\beta},\]
and use that  $\sup_{i=1}^n \mathbb E_{\bs X} Q_i = o(n^{\kappa})$.
For all $i\in\{1, \ldots, n\}$, we bound $X_i^{\mathbb E_{\bs X}Q_i}$ from above and below by \[X_i^{\sup_{i=1..n}\mathbb E_{\bs X}Q_i}\leq X_i^{\mathbb E_{\bs X}Q_i}\leq 1.\]
Plugging these bounds into~\eqref{eq:blabla} we get the following lower and upper bound 
 for $\sum_{i=1}^n \mathbb P_{\bs X}\left(Q_i - \mathbb E_{\bs X} Q_i \geq x n^{\kappa}\right)$
with $\sigma_n :=\sup_{i=1..n}\mathbb E_{\bs X}Q_i$ in the lower bound and $\sigma_n:=0$ in the upper bound,
\begin{align*}
&\alpha_2 (1\pm \varepsilon)^2
\sum_{i=1}^n \sum_{k\geq xn^{\kappa}} k^{-\beta} \frac{X_i^{k+\sigma_n}}{\Phi(X_i)}\\
&\hspace{1cm}\approx \alpha_2 (1\pm \varepsilon)^2 
\sum_{k\geq xn^{\kappa}} k^{-\beta} \sum_{i=1}^n \frac{X_i^{k+ \sigma_n}}{\Phi(X_i)}\\
&\hspace{1cm}\approx \alpha_2 (1\pm \varepsilon)^2 
\Bigg(\sum_{k= xn^{\kappa}}^{\frac{n^{\nicefrac1{\gamma}}}{\log n}} n k^{-\beta-\gamma} U_{k+\sigma_n}^{(n)}
+ \sum_{k=\frac{n^{\nicefrac1{\gamma}}}{\log n}}^{n^{\nicefrac1{\gamma}}\log n} k^{-\beta} \sum_{i=1}^n \frac{X_i^{k+\sigma_n}}{\Phi(X_i)}
+ \sum_{k=n^{\nicefrac1{\gamma}}\log n}^{\infty} k^{-\beta} \left(X_n^{_{(1)}}\right)^{k+\sigma_n} V_{k+\sigma_n}^{(n)}\Bigg),
\end{align*}
using Lemma~\ref{lem:approx_sum} notations.
Using that $\sigma_n=o(\nicefrac{n^{1/\gamma}}{\log n})$ it can be checked easily
that the second and third terms are $o(1)$--terms, independent of~$x$.
Thus only the first term of the above sum needs to be considered.
Note that there exists two integers $m_n, M_n \in [xn^{\kappa}, \nicefrac{n^{\nicefrac1{\gamma}}}{\log n}]$ such that
$U_{m_n}^{(n)}\leq U_k^{(n)}\leq U_{M_n}^{(n)}$, for all $n\geq 1$ and 
$\smash k\in [xn^{\kappa}, \nicefrac{n^{\nicefrac1{\gamma}}}{\log n}]$. 
In view of Lemma~\ref{lem:approx_sum}$(ii)$, we have 
\smash{$U_{M_n}^{\ssup{n}}\sim U_{m_n}^{\ssup{n}}\sim \alpha_1 \Gamma(1+\gamma)$} as $n\to\infty$.
Thus,
\[\sum_{k= xn^{\kappa}}^{\frac{n^{\nicefrac1{\gamma}}}{\log n}} n k^{-\beta-\gamma} U_{m_n}^{(n)}
\leq \sum_{k= xn^{\kappa}}^{\frac{n^{\nicefrac1{\gamma}}}{\log n}} n k^{-\beta-\gamma} U_k^{(n)}
\leq \sum_{k= xn^{\kappa}}^{\frac{n^{\nicefrac1{\gamma}}}{\log n}} n k^{-\beta-\gamma} U_{M_n}^{(n)},
\]
both bounds being then equivalent to $\alpha_1\Gamma(\gamma +1) n (xn^{\kappa})^{1-\beta-\gamma} \sim \alpha_1\Gamma(\gamma+1)x^{-\nicefrac1{\kappa}}$ when $n$ tends to infinity. We eventually get that, for all $n$ large enough,
\[\sum_{i=1}^n \mathbb P_{\bs X}\left(Q_i - \mathbb E_{\bs X} Q_i \geq x n^{\kappa}\right) \approx \frac{\alpha_1\alpha_2 \Gamma(\gamma+1) (1\pm \varepsilon)^2}{x^{\nicefrac1{\kappa}}},\]
which implies~\eqref{eq:cond_stable_th1} with $C_1:= \alpha_1\alpha_2 \Gamma(\gamma+1)$.
Finally, for all large enough $n$,
\begin{align*}
\frac1{n^{2\kappa}}
\sum_{i=1}^n \Var_{\bs X}\big((Q_i-\mathbb E_{\bs X} Q_i) & \indi\{|Q_i-\mathbb E_{\bs X} Q_i|<\varepsilon n^{\kappa}\}\big)
\leq \cst. n^{-2\kappa}
\sum_{i=1}^n \sum_{k\leq 2\varepsilon n^{\kappa}} (k-\mathbb E_{\bs X} Q_i)^2 k^{-\beta} X_i^k\\
&\leq \cst. n^{-2\kappa} 
\sum_{i=1}^n \sum_{k\leq \mathbb E_{\bs X} Q_i} (\mathbb E_{\bs X} Q_i)^2 k^{-\beta} X_i^k
+ \cst. n^{-2\kappa} 
\sum_{k=0}^{2\varepsilon n^{\kappa}} k^{2-\beta} \sum_{i=1}^n X_i^k\\
&\leq \cst. n^{-2\kappa} 
\sum_{i=1}^n G(X_i)^2 
+ \cst. n^{1-2\kappa}
\sum_{k=0}^{2\varepsilon n^{\kappa}} k^{2-\beta-\gamma},
\end{align*}
 in view of Lemma~\ref{lem:approx_sum}$(ii)$ and $(iii)$.
Recall that $G$ is bounded if $\beta>2$, has exponential tails if $\beta=2$, and has 
tails of polynomial order $-\sfrac{\gamma}{2-\beta}$ if $\beta <2$. Hence $\sum_{i=1}^n G(X_i)$ is $O(n)$ if $\gamma>2(2-\beta)$, and $O_{\PP}(n^{\nicefrac{2(2-\beta)}{\gamma}})$
otherwise. From this we derive that the first term above goes to zero as $n$ goes to infinity. Moreover,  the second term is a constant multiple of $\varepsilon^{3-\beta-\gamma}$, which
verifies~\eqref{eq:cond_stable_th3}  and completes the proof of~$(i)$.
\end{proof}

\section{The condensation effect}\label{sec:condensation}

In this section we not only prove Theorem~\ref{th:condensation} but also provide crucial information about the position of the condensate, which will enter into the proofs of our main theorems.
\medskip

We choose $\delta_n\downarrow 0$ such that
$\mathbb P_{\bs X}(|S_n- n \nu_n| \leq \sfrac12 n\delta_n) \to 1$, in $\mathtt P$-probability.
With $\kappa=\max\{\frac12, \frac1{\beta+\gamma-1}\}$ we can achieve this for a sequence satisfying
$n^{\kappa} \ll n\delta_n$.  If $1<\gamma<2$ we make the stronger assumption that
\smash{$n^{\nicefrac1\gamma}\ll n \delta_n$}.
%In total: $\delta_n$ verifies:
%\begin{equation}\label{eq:cond_delta_n}
%\max(n^{\nicefrac12}, n^{\nicefrac1{\gamma}})\ll \delta_n n.
%\end{equation}
%Lemma~\ref{lem:LLN_nu} gives that 
%$\nu_n \to \rho^{\star} := \mathtt{E}\frac{X \Phi'(X)}{\Phi(X)}$, almost surely. 
%Assume that $\rho\geq \rho^{\star}$ and fix $\rho-\rho^{\star} \geq \varepsilon >0$.
%Note that, asymptotically when $n$ tends to infinity, 
%$\frac{m}{n} - \nu_n -\delta_n \to \rho-\rho^{\star}>\varepsilon$. 
%Thus, for large enough $n$ (and we will assume in the following that $n$ is large enough), 
%we have $m - \nu_n n - \delta_n n = (\rho_n -\nu_n -\delta_n)n > \varepsilon n$.
We assume $\beta+\gamma>2$, $\rho>\rho^\star$ and fix $\varepsilon>0$ such that 
$\varepsilon<\frac{\beta+\gamma-2}{\beta+\gamma}\, (\rho-\rho^{\star})$ if $\gamma\leq 1$,
and %$\frac12(\rho-\rho^{\star})<\varepsilon<\rho-\rho^{\star}$ 
$\varepsilon<\frac{\beta-1}{\beta\gamma}\, (\rho-\rho^{\star})$ if $\gamma> 1$. 
\medskip

We partition the event $\{S_n=m\}$ into four disjoint events,
\[\begin{array}{ll}
\mathcal{E}_1 &= \{S_n=m, \, \exists i\in\{1, \ldots, n\} 
\text{ such that } |Q_i - (m-\nu_n n)|\leq \delta_n n, \text{ and }\forall j\neq i, Q_j\leq \varepsilon n\},\\
\mathcal{E}_2 &= \{S_n=m, \,\exists i\neq j\in\{1, \ldots, n\} 
\text{ such that } |Q_i - (m-\nu_n n)|\leq \delta_n n \text{ and } Q_j> \varepsilon n\},\\
\mathcal{E}_3 &= \{S_n=m, \,\forall i\in\{1, \ldots, n\}, |Q_i - (m-\nu_n n)| > \delta_n n
\text{ and }\exists j\in\{1, \ldots, n\}\text{ such that }Q_j > \varepsilon n\},\\
\mathcal{E}_4 &= \{S_n=m \mbox{ and, for all }i\in\{1, \ldots, n\}, Q_i\leq \varepsilon n\}.
\end{array}\]
The idea is to prove that, asymptotically as $n$ tends to infinity, $\mathcal{E}_1$ is the dominating event.
We further define the following events, for all $i,j\in\{1, \ldots, n\}$,
\[\begin{array}{ll}
\mathcal{E}_{1,i} &= \{S_n=m,\, |Q_i - (m-\nu_n n)|\leq \delta_n n \text{ and } Q_j\leq \varepsilon n \text{ for all }j\neq i\},\\
\mathcal{E}^*_{1,i} &= \{S_n=m \text{ and } |Q_i - (m-\nu_n n)|\leq \delta_n n\},\\
\mathcal{E}^*_{3,i} &= \{S_n=m, \,|Q_i - (m-\nu_n n)| > \delta_n n \text{ and }Q_i>\varepsilon n\},\\
\mathcal{D}_{i,j} &= \{S_n = m, \,|Q_i - (m-\nu_n n)|\leq \delta_n n \text{ and } Q_j>\varepsilon n\}.
\end{array}
\]

Recall that $u_n  \approx f(v_n, \mp \delta_n)$ means that
$f(v_n,\delta_n) \leq u_n \leq f(v_n, - \delta_n)$ for all sufficiently large~$n$.
\smallskip

\begin{lem}\label{lem:calcul_proba_E1i*}
For all $i\in\{1, \ldots, n\}$, with high $\mathtt{P}$-probability,
\[\mathbb{P}_{\bs X} (\mathcal{E}_{1,i}^*)
\approx \alpha_2 (\rho-\rho^{\star})^{-\beta}  n^{-\beta}  \frac{X_i^{(\rho_n-\nu_n \mp \delta_n)n}}{\Phi(X_i)}(1+o(1)),\]
with an error $o(1)$ which is uniform in~$i$.
%where the symbol $\approx$ together means the following lower and upper bounds:
%\[\alpha (\rho-\rho^{\star})^{-\beta}  n^{-\beta}  \frac{X_i^{(\rho_n-\nu_n +\delta_n)n}}{\Phi(X_i)}(1+o(1))
%\leq \mathbb{P}_{\bs X} (\mathcal{E}_{1,i}^*)
%\leq \alpha (\rho-\rho^{\star})^{-\beta}  n^{-\beta}  \frac{X_i^{(\rho_n-\nu_n - \delta_n)n}}{\Phi(X_i)}(1+o(1)),\]
\end{lem}

\begin{proof}
For all $i\in\{1, \ldots, n\}$, we  denote $\displaystyle S_{n-1}^{(i)} = \sum_{\heap{j=1}{j\neq i}}^n  Q_j$. Hence
\begin{align*}
\mathbb{P}_{\bs X} (\mathcal{E}_{1,i}^*)
&= \sum_{|k-(m-\nu_n n)|\leq \delta_n n} \mathbb{P}_{\bs X} (Q_i=k \text{ and }S_n=m)\\
&= \sum_{|k-(m-\nu_n n)|\leq \delta_n n} \mathbb{P}_{\bs X} (Q_i=k) \mathbb{P}_{\bs X}\Big(\sum_{\heap{j=1}{j\neq i}}^n Q_j = m-k\Big)\\
&= \sum_{|k-(m-\nu_n n)|\leq \delta_n n} \frac{p_k X_i^{k}}{\Phi(X_i)} \, \mathbb{P}_{\bs X}(S_{n-1}^{(i)}=m-k).
\end{align*}
For all integers $k$ such that $|k-(m-\nu_n n)|\leq \delta_n n$, we have $p_k\sim \alpha_2 (m-\nu_n n)^{-\beta}$ as~$n\uparrow\infty$.
Thus,
\begin{align*}
\mathbb{P}_{\bs X} (\mathcal{E}_{1,i}^*)
&= \sum_{|k-(m-\nu_n)n|\leq \delta_n n} \alpha_2 (m-\nu_n n)^{-\beta} \frac{X_i^{k}}{\Phi(X_i)} \mathbb{P}_{\bs X}(S_{n-1}^{(i)}=m-k) \, (1+o(1))\\
&\approx \alpha_2 (m-\nu_n n)^{-\beta} \frac{X_i^{m-\nu_n n \mp \delta_n n}}{\Phi(X_i)} 
\, \mathbb{P}_{\bs X}(|S_{n-1}^{(i)}-\nu_n n|\leq \delta_n n) \,  (1+o(1)).
\end{align*}
As the tails $\mathbb{P}_{\bs X}(Q_i>x)$ are going to zero uniformly in ${\bs X}$ we have that
$Q_i= o(n\delta_n)$ in $\mathbb{P}_{\bs X}$-probability.  Hence 
$\mathbb{P}_{\bs X}(|S_{n-1}^{_{(i)}}-\nu_n n|\leq \delta_n n)
=\mathbb{P}_{\bs X}(|S_n-\nu_n n - Q_i|\leq \delta_n n)$ is bounded from below by 
$\mathbb{P}_{\bs X}(|S_n-\nu_n n|\leq \frac12 \delta_n n)-o(1)$, where the $o$-term is independent of $i$, and this bound converges to one by choice of $\delta_n$. 
This implies the statement.
\end{proof}

\begin{lem}\label{lem:calcul_proba_Dij}
For all $i\neq j\in\{1, \ldots, n\}$, with high $\mathtt{P}$-probability,
\[\mathbb{P}_{\bs X} (\mathcal{D}_{i,j}) 
= O(n^{-2\beta}) X_i^{m-\nu_n n -\delta_n n} \sum_{k>\varepsilon n} X_j^k,\]
where the implied constant is independent of $i$ and $j$.
\end{lem}

\begin{proof}

For all $i\neq j\in\{1, \ldots, n\}$, abbreviating again $\displaystyle S_{n-1}^{(i)} = \sum_{\heap{j=1}{j\neq i}}^n  Q_j$, we have
\begin{align*}
\mathbb{P}_{\bs X} (\mathcal{D}_{i,j})
&= \sum_{|k - (m-\nu_n n)|\leq \delta_n n} \mathbb{P}_{\bs X} (Q_i=k, Q_j>\varepsilon n \text{ and }S_n=m)\\
&= \sum_{|k - (m-\nu_n n)|\leq \delta_n n} \frac{p_k X_i^k}{\Phi(X_i)} \, \mathbb{P}_{\bs X} (Q_j>\varepsilon n \text{ and }S^{(i)}_{n-1}=m-k).
\end{align*}
We now use that $0<p_0\leq \Phi(z)$, for all $z\geq 0$, together with the asymptotic behaviour of $(p_k)$ to bound this by a constant multiple of
\begin{align*}
n^{-\beta} X_i^{m-\nu_n n-\delta_n n} \,  
\mathbb{P}_{\bs X} (Q_j>\varepsilon n \text{ and }|S^{(i)}_{n-1}-\nu_n n|\leq \delta_n n)
\leq \cst. n^{-2\beta} X_i^{m-\nu_n n-\delta_n n} \sum_{k>\varepsilon n} X_j^{k},
\end{align*}
as required.
\end{proof}

\begin{lem}\label{lem:calcul_proba_E1}
If $\beta+\gamma>2$, then, with high $\mathtt{P}$-probability,
\[\mathbb{P}_{\bs X}(\mathcal E_1) 
= \left[\sum_{i=1}^n \mathbb{P}_{\bs X}(\mathcal E_{1,i}^*)\right] (1+o(1)).\]
\begin{enumerate}[(i)]
\item 
Moreover, if $\gamma>1$,
\[\mathbb{P}_{\bs X}(\mathcal E_1)
= \mathbb{P}_{\bs X}(\mathcal{E}_{1,J_n}^*) (1+o(1))
\approx  \alpha_2(\rho-\rho^{\star})^{-\beta} \big(X^{{(1)}}_n\big)^{(\rho_n-\nu_n\mp \delta_n)n} n^{-\beta} (1+o(1)),\]
where $J_n\in\{1,\ldots,n\}$ is the index realising the maximum fitness, i.e.\ $X_{J_n}=X^{{(1)}}_n$.
%\item If $\gamma >1$, for all sequences $(\omega_n)_{n\geq 1}$ such that $n^{\nicefrac1{\gamma}}=o(\omega_n)$,
%$\displaystyle\mathbb P_{\bs X}(\mathcal E_1) \geq \cst.\ n^{-\beta} \big(X_n^{_{(1)}}\big)^{(\rho_n-\nu_n)n+\omega_n}.$
\begin{enumerate}[(a)]
\item If $\gamma \geq2$, we have
$\displaystyle\mathbb P_{\bs X}(\mathcal E_1) \geq \cst.\ n^{-\beta} \big(X_n^{_{(1)}}\big)^{(\rho_n-\nu_n)n}.$
\item If $1<\gamma<2$, then, for all $\omega_n$ such that $n^{\nicefrac1\gamma}\ll \omega_n \ll n\delta_n$, we have
$\displaystyle\mathbb P_{\bs X}(\mathcal E_1) \geq \cst.\ n^{-\beta} \big(X_n^{_{(1)}}\big)^{(\rho_n-\nu_n)n+\omega_n}.$
\end{enumerate}

\item If $\gamma < 1$, then $\mathbb{P}_{\bs X}(\mathcal E_1) =  (1+o(1)) \, \alpha_1\alpha_2 (\rho-\rho^\star)^{-\beta-\gamma} \Gamma(\gamma+1) n^{1-\beta-\gamma}$.

If $\gamma =1$, then $\mathbb{P}_{\bs X}(\mathcal E_1) = \Theta_{\PP}(n^{-\beta})$.
\end{enumerate}
\end{lem}

\begin{proof}
First note that, by definition of the events $\mathcal{E}_{1,i}^*$, $\mathcal{E}_1$ and $\mathcal{D}_{i,j}$,
\[\sum_{i=1}^n \mathbb{P}_{\bs X} (\mathcal{E}_{1,i}^*) - \sum_{i\neq j}\mathbb{P}_{\bs X} (\mathcal{D}_{i,j})
\leq \mathbb{P}_{\bs X}(\mathcal{E}_1)
\leq \sum_{i=1}^n \mathbb{P}_{\bs X} (\mathcal{E}_{1,i}^*).\]
Our aim is to prove that $\sum_{i\neq j}\mathbb{P}_{\bs X} (\mathcal{D}_{i,j})$ is negligible in front of $\sum_{i=1}^n \mathbb{P}_{\bs X} (\mathcal{E}_{1,i}^*)$. In view of Lemmas~\ref{lem:calcul_proba_E1i*} and~\ref{lem:calcul_proba_Dij},
we have
\[\sum_{i=1}^n \mathbb{P}_{\bs X} (\mathcal{E}_{1,i}^*)
\approx \alpha_2 (\rho-\rho^{\star})^{-\beta} n^{-\beta} \sum_{i=1}^n \frac{X_i^{(\rho_n-\nu_n\mp \delta_n)n}}{\Phi(X_i)}
(1+o(1)),\]
where the $o(1)$-term is independent of $i$,
and
\[\sum_{i\neq j}\mathbb{P}_{\bs X} (\mathcal{D}_{i,j})
\leq \cst. n^{-2\beta} \sum_{i\neq j} X_i^{(\rho_n-\nu_n-\delta_n)n} \sum_{k>\varepsilon n} X_j^k.\]
It is thus enough to prove that the ratio
\[\Delta_n :=
\frac{n^{-\beta} \sum_{i\neq j} X_i^{(\rho_n-\nu_n-\delta_n)n} \sum_{k>\varepsilon n} X_j^k}
{\sum_{i=1}^n X_i^{(\rho_n -\nu_n + \delta_n)n}}\]
tends to zero in $\PP$-probability, as~$n\uparrow\infty$.

\vspace{\baselineskip}
\noindent{\bf $\bs{(i)}$ Assume $\bs{\gamma > 1}$.} In this case, in view of Lemma~\ref{lem:extremal_properties_X} and Lemma~\ref{lem:approx_sum}$(i)$, we have
\[\Delta_n
\leq \cst. n^{-\beta}\ 
\frac{\big(X_n^{_{(1)}}\big)^{(\rho_n-\nu_n-\delta_n)n} \sum_{k>\varepsilon n} \big(X_n^{_{(1)}}\big)^k V_k^{(n)}}
{(X_n^{_{(1)}})^{(\rho_n-\nu_n + \delta_n)n}}
=O_{\PP} \big(n^{1/\gamma-\beta}\big) (X_n^{_{(1)}})^{(\varepsilon - 2\delta_n)n} = O_{\PP} \big(n^{1/\gamma-\beta} \big),\]
which tends to 0 when $n\uparrow\infty$.
%Also, if $\gamma = 1$, then, \smash{$\Delta_n= O(n^{1-\beta}) \to 0$} since $\beta >1$.
%\texttt{This looks wrong, isn't this exponent $3-\beta$?}
We now prove that $ \sum_{i=1}^n \mathbb{P}_{\bs X} (\mathcal{E}_{1,i}^*)\sim\mathbb{P}_{\bs X}(\mathcal{E}_{1,J_n}^*)$. 
It is enough to prove that
\[\Lambda_n:=
\frac{\sum_{i\neq J_n} \mathbb{P}_{\bs X}(\mathcal{E}_{1,i}^*)}{\mathbb{P}_{\bs X}(\mathcal{E}_{1,J_n}^*)}
\to 0 \quad \mbox{ as }  n\uparrow\infty.\]
We have, in view of Lemma~\ref{lem:calcul_proba_E1i*} and Lemma~\ref{lem:extremal_properties_X}, for sufficiently large~$n$,
\begin{align*}
\Lambda_n  &\leq p_0^{-1} \frac
{n^{-\beta}\sum_{i\neq J_n} X_i^{(\rho_n-\nu_n - \delta_n)n}}
{n^{-\beta}(X_n^{_{(1)}})^{(\rho_n-\nu_n + \delta_n)n}}\; (1+o(1))
 \leq p_0^{-1} \frac
{n (X_n^{_{(2)}})^{(\rho_n-\nu_n - \delta_n)n}}
{(X_n^{_{(1)}})^{(\rho_n-\nu_n+ \delta_n)n}}\; (1+o(1))\\
&
\leq  \mathtt{cst.}  n \Big(\frac{X_n^{_{(2)}}}{X_n^{_{(1)}}}\Big)^{(\rho_n-\nu_n-\delta_n)n} \big(X_n^{_{(1)}}\big)^{-2\delta_n n}
\leq  \mathtt{cst.} n \big(1-\Theta_{\PP}( n^{-\nicefrac1{\gamma}})\big)^{(\rho_n-\nu_n-3\delta_n)n}.
\end{align*}
This implies $\Lambda_n 
\leq   \mathtt{cst.} n \exp\left(-(\rho_n-\nu_n-3\delta_n) \Theta_{\PP}(n^{1-\nicefrac1{\gamma}})\right)
\to 0,$
as $n\uparrow\infty$, concluding the proof of~$(i)$.
\vspace{\baselineskip}

{\bf $\bs{(a)}$ Assume $\bs{\gamma \geq 2}$.}
We have $\mathbb P_{\bs X}(\mathcal E_1)
\sim \mathbb P_{\bs X}(\mathcal E^*_{1,J_n}),$ where $J_n$ is the index of the largest fitness.
Moreover,
\begin{align*}
\mathbb P_{\bs X}(\mathcal E^*_{1,J_n})
& = \sum_{^{|k-(\rho_n-\nu_n)n|\leq \delta_n n}} \frac{p_k \big(X_n^{_{(1)}}\big)^k}{\Phi\big(X_n^{_{(1)}}\big)} 
\, \mathbb P_{\bs X}\Big(\sum_{\heap{i=1}{i\neq J_n}}^n Q_i = m-k \Big)
\geq  \sum_{\heap{(\rho_n-\nu_n - \delta_n) n}{\,\,\leq k\leq (\rho_n-\nu_n)n}} 
p_k \big(X_n^{_{(1)}}\big)^k \, \mathbb P_{\bs X}\Big(\sum_{\heap{i=1}{i\neq J_n}}^n Q_i = m-k \Big)\\
&\geq \cst. n^{-\beta} \big(X_n^{_{(1)}}\big)^{(\rho_n-\nu_n)n} \, \mathbb P_{\bs X}\Big(0\leq \sum_{\heap{i=1}{i\neq J_n}}^n Q_i-\nu_n n \leq \delta_n n\Big).
\end{align*}
Recall that \smash{$\nicefrac{Q_{J_n}}{n\delta_n}$} goes to zero in $\P_{\bs X}$-probability,  and hence, by the grand canonical central limit theorem with a normal limit, 
the probability above goes to $\nicefrac12$. Therefore we get
\smash{$\mathbb P_{\bs X}(\mathcal E^*_{1,J_n})
\geq \cst. n^{-\beta} \big(X_n^{_{(1)}}\big)^{(\rho_n-\nu_n)n}.$}
\vspace{\baselineskip}

{\bf $\bs{(b)}$ Assume $\bs{1<\gamma < 2}$.} Let $n^{\nicefrac1\gamma}\ll \omega_n \leq n\delta_n$, then, as above
\begin{align*}
\mathbb P_{\bs X}(\mathcal E^*_{1,J_n})
%&
%= \sum_{^{|k-(\rho_n-\nu_n)n|\leq \delta_n n}} \frac{p_k \big(X_n^{_{(1)}}\big)^k}{\Phi\big(X_n^{_{(1)}}\big)} 
%\, \mathbb P_{\bs X}\Big(\sum_{\heap{i=1}{i\neq J_n}}^n Q_i = m-k \Big)
%\geq  \sum_{\heap{(\rho_n-\nu_n - \delta_n) n}{\,\,\leq k\leq (\rho_n-\nu_n)n+\omega_n}} 
%p_k \big(X_n^{_{(1)}}\big)^k \, \mathbb P_{\bs X}\Big(\sum_{\heap{i=1}{i\neq J_n}}^n Q_i = m-k \Big)\\
&\geq \cst.\ n^{-\beta} \big(X_n^{_{(1)}}\big)^{(\rho_n-\nu_n)n+\omega_n} \, 
\mathbb P_{\bs X}\Big(-\omega_n\leq \sum_{\heap{i=1}{i\neq J_n}}^n Q_i-\nu_n n \leq \delta_n n\Big).
\end{align*}
Note that $n^{\kappa}\leq n^{\nicefrac1\gamma}\ll \omega_n$ 
where $\kappa = \max\{\frac1{2}, \frac1{\beta+\gamma-1}\}$. Thus, in view of Proposition~\ref{prop:TCL}, we have
\[\mathbb P_{\bs X}\Big(-\omega_n\leq \sum_{\heap{i=1}{i\neq J_n}}^n Q_i-\nu_n n \leq \delta_n n\Big) \to 1,\]
when $n$ goes to infinity, implying the statement.

\vspace{\baselineskip}

{\bf $\bs{(ii)}$ Assume $\bs{\gamma \leq 1}$ and $\bs{\beta+\gamma >2}$.} 
We have, in view of Lemma~\ref{lem:approx_sum}$(ii)$ and $(iii)$, that
$\sum_{j=1}^n X_j^{\varepsilon n}$ is of order $\Theta_{\PP}(n^{1-\gamma})$ if $\gamma <1$, and of order $o(\log n)$ if $\gamma =1$.
Therefore,
\[\Delta_n
\leq \cst.
\begin{cases}
\displaystyle
n^{-\beta} \,
\frac
{[(\rho_n-\nu_n-\delta_n)n]^{-\gamma} \sum_{k>\varepsilon n} nk^{-\gamma} U_k^{(n)}}
{[(\rho_n-\nu_n+\delta_n)n]^{-\gamma} }
=O_{\PP}(n^{2-\beta-\gamma})
&\text{ if }\gamma <1,\\
& \\
\displaystyle
n^{1-\beta}\,
\frac
{\sum_{i=1}^n X_i^{(\rho_n-\nu_n-\delta_n)n} \sum_{j=1}^n X_j^{\varepsilon n}}
{\sum_{i=1}^n X_i^{(\rho_n-\nu_n + \delta_n)n}}
=o(n^{1-\beta} \log^2 n)
&\text{ if }\gamma =1,
\end{cases}
\]
Hence $\Delta_n\to 0$ in $\PP$-probability
 if $\gamma<1$ and
$\beta + \gamma >2$, or if $\gamma=1$.  Moreover, we have
\[\mathbb P_{\bs X}(\mathcal E_1)
= \mathbb P_{\bs X} \left(\bigcup_{i=1}^n \mathcal E_{1,i}^*\right)\ (1+o(1))
\approx \alpha_2 (\rho-\rho^{\star})^{-\beta} n^{-\beta} \sum_{i=1}^n \frac{X_i^{(\rho_n-\nu_n\mp \delta_n)n}}{\Phi(X_i)}\ (1+o(1)).\]
using Lemma~\ref{lem:approx_sum}$(ii)$ if $\gamma<1$, and Lemma~\ref{lem:approx_sum}$(iii)$ if $\gamma=1$ concludes the proof.
\end{proof}

\begin{lem}\label{lem:E2<<E1}
If $\beta+\gamma>2$, then with high $\mathtt{P}$-probability, $\mathbb{P}_{\bs X}(\mathcal E_2) \ll \mathbb{P}_{\bs X}(\mathcal E_1)$.
\end{lem}

\begin{proof}
Note that $\mathbb P_{\bs X}(\mathcal E_2) = \sum_{i\neq j} \mathbb P_{\bs X} (\mathcal D_{i,j})$, and we have already shown in the proof of
Lemma~\ref{lem:calcul_proba_E1} that this sum is negligible in front of $\mathbb P_{\bs X}(\mathcal E_1)$.
\end{proof}

\begin{lem}\label{lem:E4<<E1}
If $\beta+\gamma>2$, then, with high $\mathtt{P}$-probability, $\mathbb{P}_{\bs X}(\mathcal E_4) \ll \mathbb{P}_{\bs X}(\mathcal E_1)$.
\end{lem}

\begin{proof}
We define the truncated variables 
$\bar Q_i:= Q_i \indi\{Q_i\leq \varepsilon n\}$ and 
$\bar S_n = \sum_{i=1}^n \bar Q_i$.
As $\mathcal{E}_4 \subset \{\bar S_n = m\}$, we have
\[\mathbb P_{\bs X}(\mathcal{E}_4) 
\leq \mathtt{e}^{-sm} \mathbb{E}_{\bs X} \big[ \mathtt{e}^{s\bar S_n} \big]
= \mathtt{e}^{-sm} \prod_{i=1}^n \mathbb{E}_{\bs X}\big[ \mathtt{e}^{s\bar Q_i }\big], 
\qquad \mbox{for every $s>0$.}\]
There exist two constants $K_1, K_2>0$, such that
\[
\mathbb{E}_{\bs X} \mathtt{e}^{s\bar Q_i}
\leq 1+s \mathbb E_{\bs X} \bar Q_i + \sum_{k=1}^{\varepsilon n} \frac{p_k X_i^k}{\Phi(X_i)} \,  (\mathtt{e}^{sk}-1-sk)
\leq  1+ s \mathbb E_{\bs X} Q_i
+ K_1 \sum_{k=1}^{\nicefrac{2\beta}{s}} k^{-\beta} X_i^k (sk)^2
+ K_2 \sum_{k=\nicefrac{2\beta}{s}}^{\varepsilon n} k^{-\beta}X_i^k \mathtt{e}^{sk}.
\]
Allowing $s$ to depend on~$n$, we define, for any sequence $(s_n)$,  the quantities
\[S^{{(1)}}_n := \sum_{i=1}^n \sum_{k=1}^{\nicefrac{2\beta}{s_n}} k^{-\beta} X_i^k (s_nk)^2,\quad
\text{ and }\quad
S_n^{{(2)}}:= \sum_{i=1}^n \sum_{k=\nicefrac{2\beta}{s_n}}^{\varepsilon n} k^{-\beta}X_i^k \mathtt{e}^{s_nk}.\]
We then have
\[
\mathbb{P}_{\bs X}(\mathcal{E}_4)
\leq \exp\big(-s_nm + s_nn\nu_n  + K_1 S_n^{{(1)}} + K_2 S_n^{{(2)}}\big).\]

\vspace{\baselineskip}
\noindent{\bf $\bs{(i)}$ The case $\bs{\gamma\leq 1}$ and $\bs{\beta+\gamma>2}$.}
We fix $s_n:=a\frac{\log n}{n}$, where \smash{$a=\frac{\beta+\gamma}{\rho-\rho^\star}$.}
We first prove that $$S_n^{(1)}= o(ns_n) \qquad \mbox{ as $n\uparrow\infty$.}$$
In view of Lemma~\ref{lem:approx_sum}$(ii)$ and $(iii)$, using that $\nicefrac{2\beta}{s_n} =  o(n^{\nicefrac1{\gamma}})$, we have
\[S^{(1)}_n
= s_n^2 \sum_{k=1}^{\nicefrac{2\beta}{s_n}} k^{2-\beta} \sum_{i=1}^n X_i^k
= n s_n^2 \sum_{k=1}^{\nicefrac{2\beta}{s_n}} k^{2-\beta-\gamma} U_k^{(n)}
\leq \cst.\ ns_n^2 \sum_{k=1}^{\nicefrac{2\beta}{s_n}} k^{2-\beta-\gamma},\]
from which we infer that 
%\[S_n^{(1)}= \begin{cases}
%O(n s^2) = o(ns) & \text{ if }\beta+\gamma>3\\
%\displaystyleO\left(n s^2 \log \frac{1}{s}\right) = o(ns) & \text{ if }\beta+\gamma=3\\
%\displaystyleO\left(n s^2 \frac1{s^{3-\beta-\gamma}}\right) = O(n s^{\beta+\gamma-1}) = o(ns) & \text{ if }\beta+\gamma < %3,
%\end{cases}\]
%which implies that, in all three cases, 
$S_n^{(1)} = o(ns_n)$. %, using $\beta+\gamma >2$.
Next, we prove that
$$S_n^{(2)}= o(n s_n) \qquad \mbox{ as $n\uparrow\infty$.}$$
Denote by $u_k:= k^{-\beta}X_i^k \mathtt{e}^{s_nk}$. Observe that, for all $k\geq \frac{2\beta}{s_n}$, we have
$\frac{u_k}{u_{k+1}}\leq \frac{\mathtt{e}^{-\nicefrac{s_n}{2}}}{X_i},$ and thus 
$$u_k \leq \left(\frac{\mathtt{e}^{-\nicefrac{s_n}{2}}}{X_i}\right)^{\lfloor\varepsilon n\rfloor - k} u_{\lfloor\varepsilon n\rfloor}.$$
This implies that 
\[
\sum_{k=\nicefrac{2\beta}{s_n}}^{\varepsilon n} k^{-\beta}X_i^k \mathtt{e}^{s_nk}
\leq \lfloor\varepsilon n\rfloor^{-\beta} \, \mathtt{e}^{s_n \lfloor\varepsilon n\rfloor}
\sum_{k=\nicefrac{2\beta}{s_n}}^{\varepsilon n} X_i^k \left(\mathtt{e}^{-\nicefrac{s_n}{2}}\right)^{\lfloor\varepsilon n\rfloor - k}
\leq \cst.n^{a\varepsilon-\beta} \;\frac{X_i^{\nicefrac{2\beta}{s_n}}}{1-\mathtt{e}^{-\nicefrac{s_n}{2}}}.
%\leq \cst.n^{a\varepsilon-\beta}\, \lfloor\varepsilon n\rfloor^{-\beta} \;\frac{X_i^{\nicefrac{2\beta}{s_n}}}{s_n},
\]
%since $s_n\to 0$ when $n$ tends to infinity.
Therefore, using $1-\mathtt{e}^{-\nicefrac{s_n}{2}} \geq \nicefrac{s_n}{2}$, and 
Lemma~\ref{lem:approx_sum}$(ii)$ in conjunction with  $\nicefrac{2\beta}{s_n} \ll n^{\nicefrac1{\gamma}}$, we get
\[
S_n^{(2)}
\leq \cst. \frac{n^{a\varepsilon-\beta}}{s_n} \sum_{i=1}^n X_i^{\nicefrac{2\beta}{s_n}}
= O_\PP\big(n^{1+a\varepsilon-\beta} s_n^{\gamma-1}\big),
\]
and, since $a\varepsilon < \beta+\gamma-2$, this  implies $S_n^{(2)}=o(ns_n)$  as required. 
%\smallskip
%
Summarising, we have shown that
$$\mathbb{P}_{\bs X}(\mathcal{E}_4)
\leq \exp(-s_nm + s_nn \nu_n + o(s_nn))
= n^{-a(\rho-\rho^{\star}+o(1))}.$$
Recall that $\mathbb{P}_{\bs X}(\mathcal{E}_1)= n^{1-\beta-\gamma} (1+o(1))$.
As $a(\rho-\rho^{\star})> \beta+\gamma-1$,
we get that $\mathbb{P}_{\bs X}(\mathcal E_4) \ll \mathbb{P}_{\bs X}(\mathcal{E}_1)$,
as $n\uparrow\infty$.

\vspace{\baselineskip}
\noindent{\bf $\bs{(ii)}$ The case $\bs{\gamma > 1}$.}
In that case, choose $s_n = -\log X_n^{_{(1)}} + \frac{a\log n}{n} = \Theta_{\PP}(n^{-\nicefrac1{\gamma}})$
for some positive $a$ satisfying
$$\sfrac{\beta}{\rho-\rho^\star}<a< \sfrac{\beta-1}{\varepsilon \gamma}.$$
We now show that $$S_n^{(1)}=o(1).$$  We have
\[S_n^{(1)}
\leq \cst.\ s_n^2 \sum_{k=1}^{\frac{2\beta}{s_n}} k^{2-\beta} \sum_{i=1}^n X_i^k
= \cst.\ s_n^2 \bigg(\sum_{k=1}^{\frac{2\beta}{s_n \log n}} k^{2-\beta} \sum_{i=1}^n X_i^k
+ \sum_{k=\frac{2\beta}{s_n \log n}}^{\frac{2\beta}{s_n}} k^{2-\beta} \sum_{i=1}^n X_i^k\bigg).
\]
Using the notation of Lemma~\ref{lem:approx_sum}$(ii)$, we get
\[S_n^{(1)}
\leq \cst.\ s_n^2 \bigg(\sum_{k=1}^{\frac{2\beta}{s_n \log n}} n k^{2-\beta-\gamma} U_k^{(n)}
+ \sum_{k=\frac{2\beta}{s_n \log n}}^{\frac{2\beta}{s_n}} k^{2-\beta} \sum_{i=1}^n X_i^{\frac{2\beta}{s_n \log n}}\bigg).
\]
There exists an integer~$M_n$ such that \smash{$\max\big\{ U_k^{(n)} \colon k\in\{1,\dots,\nicefrac{2\beta}{s_n \log n}\} \big\}  = U_{M_n}^{(n)}$}.
Using Lemma~\ref{lem:approx_sum}$(ii)$ in conjunction with $M_n\leq \nicefrac{2\beta}{s_n \log n} \ll n^{\nicefrac1{\gamma}}$, we get that \smash{$U_{M_n}^{(n)}\sim \alpha_1\Gamma(\gamma+1)$}. Thus, using again Lemma~\ref{lem:approx_sum}$(ii)$ for the second term of the sum, we get
\[
S_n^{(1)}
\leq \cst.\ n s_n^2 \bigg(  \sum_{k=1}^{\frac{2\beta}{s_n \log n}} k^{2-\beta-\gamma}
+ \left(\sfrac{2\beta}{s_n \log n}\right)^{-\gamma} \sum_{k=\frac{2\beta}{s_n \log n}}^{\frac{2\beta}{s_n}} k^{2-\beta}\bigg).
\]
Starting from this, a simple calculation gives $S_n^{(1)} = o(1)$, as claimed. We now show that
$$S_n^{(2)}=o(1).$$
To this end, recall the definition of $s_n$, then split the sum and estimate
\begin{align*}
S_n^{(2)}
&= \sum_{k=\nicefrac{2\beta}{s_n}}^{\varepsilon n} k^{-\beta} \mathtt{e}^{s_nk} \sum_{i=1}^n X_i^k
= \sum_{k=\nicefrac{2\beta}{s_n}}^{\varepsilon n} k^{-\beta} \mathtt{e}^{ak\frac{\log n}{n}} 
\sum_{i=1}^n \Big(\frac{X_i}{X_n^{_{(1)}}}\Big)^k
\\
%&= \sum_{k=\nicefrac{2\beta}{s_n}}^{n^{\nicefrac1{\gamma}}\log^2 n} k^{-\beta} \mathtt{e}^{ak\frac{\log n}{n}} 
%\sum_{i=1}^n \Big(\frac{X_i}{X_n^{_{(1)}}}\Big)^k
%+ \sum_{k=n^{\nicefrac1{\gamma}}\log^2 n}^{\varepsilon n} k^{-\beta} \mathtt{e}^{ak\frac{\log n}{n}} 
%\sum_{i=1}^n \Big(\frac{X_i}{X_n^{_{(1)}}}\Big)^k\\
&\leq V^{(n)}_{\nicefrac{2\beta}{s_n}}\, 
 \mathtt{e}^{an^{\nicefrac1{\gamma}-1}\log n} 
\sum_{k=\nicefrac{2\beta}{s_n}}^{n^{\nicefrac1{\gamma}}\log^2 n} k^{-\beta}
+ \mathtt{e}^{a\varepsilon\log n} \sum_{k=n^{\nicefrac1{\gamma}}\log^2 n}^{\varepsilon n} k^{-\beta} V_k^{(n)},
\end{align*}
using the notation and result of Lemma~\ref{lem:approx_sum}$(i)$. Using again Lemma~\ref{lem:approx_sum}$(i)$,
for all $k\geq n^{\nicefrac1{\gamma}} \log^2n$, we have \smash{$V_k^{_{(n)}}\leq V_{n^{\nicefrac1{\gamma}}\log^2n}^{_{(n)}}$} and the 
right hand side converges to one. Using also  Lemma~\ref{lem:approx_sum}$(iii)$ we get, 
\[
S_n^{(2)}
\leq o(\log n) \left(\frac{2\beta}{s_n}\right)^{1-\beta}
+ \cst.n^{a\varepsilon} \left(n^{\nicefrac1{\gamma}}\log^2 n\right)^{1-\beta}
\leq o\left(n^{\frac{1-\beta}{\gamma}} \log n\right) + o\left(n^{a\varepsilon+\frac{1-\beta}{\gamma}}\right)
= o(1).
\]
To complete the proof recall that
\[\mathbb P_{\bs X}(\mathcal E_4) \leq \exp\left(-(\rho_n-\nu_n) ns_n + o(1)\right)
= n^{-a(\rho-\rho^{\star})+o(1)} \big(X_n^{_{(1)}}\big)^{(\rho_n-\nu_n)n},\]
and $\mathbb P_{\bs X}(\mathcal E_1) \geq \texttt{cst}.n^{-\beta} (X_n^{_{(1)}})^{(\rho_n-\nu_n+\delta_n)n}$.
Therefore,
\[\frac{\mathbb P_{\bs X}(\mathcal E_4)}{\mathbb P_{\bs X}(\mathcal E_1)}
\leq n^{\beta - a(\rho-\rho^{\star})+o(1)} \big(X_n^{_{(1)}}\big)^{-\delta_n n}.\]
Since $X_n^{(1)} = 1-\Theta_\PP(n^{-\nicefrac1{\gamma}})$, we have that $(X_n^{(1)})^{-\delta_n n} = \exp(\Theta_\PP(\delta_n n^{1-\nicefrac1{\gamma}}))$. If $1<\gamma < 2$, we have  that $\delta_n n^{1-\nicefrac1{\gamma}} \to 0$, which implies
$\mathbb P_{\bs X}(\mathcal E_4) \ll \mathbb P_{\bs X}(\mathcal E_1)$ by choice of~$a$.
If $\gamma \geq 2$, we conclude the proof using the better bound for $\mathbb P_{\bs X} (\mathcal E_1)$, 
which was proved in Lemma~\ref{lem:calcul_proba_E1}$(ia)$.
\end{proof}

\begin{lem}\label{lem:E3<<E1}
If $\beta+\gamma>2$, then, with high $\mathtt{P}$-probability, $\mathbb{P}_{\bs X}(\mathcal E_3) \ll \mathbb{P}_{\bs X}(\mathcal E_1)$.
\end{lem}

\begin{proof}
$\bs{(i)}$ {\bf The case} $\bs{\gamma \leq 1}$ {\bf and } $\bs{\beta+\gamma>2}$.
In this case $\mathbb P_{\bs X}(\mathcal E_1) = \Theta_\PP(n^{1-\beta-\gamma})$.
We get
\begin{align*}
\mathbb P_{\bs X}(\mathcal E_3)
&\leq \sum_{j=1}^n \sum_{\heap{|k-(m-\nu_n n)|>\delta_n n}{k>\varepsilon n}}
\frac{p_k X_j^k}{\Phi(X_j)} \mathbb P_{\bs X}\Big(\sum_{i\neq j} Q_i=m-k \Big)\\
&\leq \cst.  (\varepsilon n)^{-\beta} \sum_{j=1}^n X_j^{\varepsilon n} \,
\mathbb P_{\bs X}(|S_{n-1}^{(j)}-\nu_n n| > \delta_n n)
=o\left(n^{1-\beta-\gamma}\right),
\end{align*}
in view of Lemma~\ref{lem:approx_sum}$(ii)$ and Lemma~\ref{lem:LLN_Q}.

\vspace{\baselineskip}
$\bs{(ii)}$ {\bf The case} $\bs{\gamma > 1}$.
We decompose the event $\mathcal E_3 \subset \mathcal E_{3,1} \cup \bigcup_{j=1}^n \mathcal E_{3,2}^{(j)}$ where
\[\mathcal E_{3,1} = \{S_n = m \ ;\ \exists i, j\in\{1, \ldots, n\}\text{ such that }
Q_i\geq m-\nu_n n+\delta_n n \text{ and }
Q_j >\varepsilon n
\}\]
and 
\[\mathcal E_{3,2}^{(j)} = \{S_n = m \text{ and } 
Q_i < m-\nu_n n-\delta_n n \ \forall i\in\{1, \ldots, n\}\text{ and } Q_j >\varepsilon n
\}, \mbox{ for $j\in\{1,\ldots,n\}$.} \]
Note that, in view of Lemma~\ref{lem:approx_sum}$(i)$ and our choice of $\delta_n$,
\begin{align*}
\mathbb P_{\bs X}(\mathcal E_{3,1})
&\leq \sum_{i=1}^n \sum_{k\geq m-\nu_n n+\delta_n n}
\frac{p_k X_i^k}{\Phi(X_i)} \mathbb P_{\bs X}\Big(\sum_{\heap{j=1}{j\neq i}}^n Q_j = m-k\Big)\\
&\leq \cst. n^{-\beta} \sum_{i=1}^n X_i^{(\rho_n-\nu_n+\delta_n) n} \mathbb P_{\bs X}\Big(\sum_{\heap{j=1}{j\neq i}}^n Q_j -\nu_n n \leq -\delta_n n\Big)
= o\big(n^{-\beta}\big)  \big(X_n^{_{(1)}}\big)^{(\rho_n-\nu_n+\delta_n) n}.
\end{align*}
Recalling the lower bound in Lemma~\ref{lem:calcul_proba_E1}$(i)$ we get 
$\mathbb P_{\bs X}(\mathcal E_{3,1}) \ll \mathbb P_{\bs X}(\mathcal E_1)$.
\smallskip

We now focus on the estimate for the events~$\mathcal E_{3,2}^{(j)}$.
We first deal with the summand $j=J_n$, the index of the site carrying the largest fitness. 
Abbreviate $c_n:= \rho_n-\nu_n-\delta_n$ and denote, for $k>\varepsilon n$, 
$$\mathcal E_{3,2}^{k} = \big\{\forall i\neq J_n\ Q_i<c_n n  \text{ and }\sum_{i\neq J_n} Q_i = m-k\big\}.$$ 
Then, letting $s_n= -\log X_n^{_{(2)}}$ 
and $\bar Q_i = Q_i \indi\{Q_i < c_n n\}$, we get from Markov's inequality
\[
\mathbb P_{\bs X}(\mathcal E_{3,2}^{k})
\leq \mathtt e^{-(m-k)s_n} \prod_{i\neq J_n} \mathbb E_{\bs X} \left[\mathtt e^{s_n \bar Q_i}\right].
\]
Observe that, for all $i\neq J_n$, we have
\begin{align*}
\mathbb E_{\bs X} \left[\mathtt e^{s_n \bar Q_i}\right]
&\leq 1+ s_n \,\mathbb E_{\bs X} \bar Q_i + \sum_{\ell < c_n n} \frac{p_{\ell} X_i^{\ell}}{\Phi(X_i)} (\mathtt e^{s_n \ell}-1-s_n \ell)\\
&\leq 1+ s_n \,\mathbb E_{\bs X} Q_i + K_1 \sum_{\ell \leq \nicefrac{2\beta}{s_n}} \ell^{-\beta} X_i^{\ell} (s_n \ell)^2
+ K_2 \sum_{\nicefrac{2\beta}{s_n} < \ell < c_n n} \ell^{-\beta} X_i^{\ell} \mathtt e^{s_n \ell},
\end{align*}
where $K_1$ and $K_2$ are two positive constants that do not depend on $i$.
Thus,
\[\mathbb P_{\bs X}(\mathcal E_{3,2}^{k})
\leq \exp\left(-(m-k-\nu_n n)s_n + K_1 S_n^{(1)} + K_2 S_n^{(2)}\right),
\]
where
\[S_n^{(1)}:= \sum_{i\neq J_n} \sum_{\ell \leq \nicefrac{2\beta}{s_n}} \ell^{-\beta} X_i^{\ell} (s_n \ell)^2
\quad \text{ and }\quad
S_n^{(2)}:= \sum_{i\neq J_n} \sum_{\nicefrac{2\beta}{s_n} < \ell < c_n n} \ell^{-\beta} X_i^{\ell} \mathtt e^{s_n \ell}.\]
Note that $S_n^{_{(1)}}$ and $S_n^{_{(2)}}$ are independent of $k$.
We have already encountered $S_n^{_{(1)}}$ in the proof of Lemma~\ref{lem:E4<<E1}, 
and proved that $S_n^{_{(1)}} = o(1)$. The sum $S_n^{_{(2)}}$ is slightly different than the one studied 
in the proof of Lemma~\ref{lem:E4<<E1}, but the same calculation yields $S_n^{_{(2)}} = O_\PP(n^{(1-\beta)/\gamma}) = o(1)$.  Summarising, we see that
\[\mathbb P_{\bs X}(\mathcal E_{3,2}^{k}) \leq \exp\big(-(m-k-\nu_n n)s_n + o(1)\big)
= \big(X_n^{_{(2)}}\big)^{(\rho_n-\nu_n)n-k}(1+o(1)),\]
where the $o(1)$-term does not depend on $k$. Thus,
\begin{align*}
\mathbb P_{\bs X}\big( \mathcal E_{3,2}^{(J_n)}\big)
& =   \sum_{\varepsilon n<k<c_n n} \mathbb P_{\bs X}\big( Q_{J_n}=k\big) \,  \mathbb P_{\bs X}(\mathcal E_{3,2}^{k})\\
& \leq \cst. \big(X_n^{_{(2)}}\big)^{(\rho_n-\nu_n)n}
\sum_{\varepsilon n<k<c_n n} k^{-\beta} \Big(\frac{X_n^{_{(1)}}}{X_n^{_{(2)}}}\Big)^k
\leq O_\PP\big( n^{\nicefrac1{\gamma}-\beta} \big) \big(X_n^{_{(2)}}\big)^{(\rho_n-\nu_n)n}  \Big(\frac{X_n^{_{(1)}}}{X_n^{_{(2)}}}\Big)^{c_n n},
\end{align*}
where we have used that $\frac{X_n^{_{(1)}}}{X_n^{_{(2)}}} = 1+\Theta_\PP(n^{-\nicefrac{1}{\gamma}})$ by Lemma~\ref{lem:extremal_properties_X}.
%Summarising, we get
%\begin{align*}
%\mathbb P_{\bs X}\big( \mathcal E_{3,2}^{(J_n)}\big)
%&\leq \cst. n^{\nicefrac1{\gamma}-\beta} \big(X_n^{_{(2)}}\big)^{\delta_n n} \big(X_n^{_{(1)}}\big)^{c_n n}.
%\end{align*}
\smallskip

Now assume that $\gamma >2$. Then, in view of the lower bound proved in Lemma~\ref{lem:calcul_proba_E1}$(ia)$, we have 
\[
\frac{\mathbb P_{\bs X}( \mathcal E_{3,2}^{(J_n)})}{\mathbb P_{\bs X}(\mathcal E_1)}
\leq O_\PP\big( n^{\nicefrac1{\gamma}} \big) \Big(\frac{X_n^{_{(2)}}}{X_n^{_{(1)}}}\Big)^{\delta_n n}
\leq O_\PP\big( n^{\nicefrac1{\gamma}} \big) \mathtt e^{-\Theta_\PP( \delta_n n^{1-\nicefrac1{\gamma}})}
= o(1),
\]
because $\delta_n n^{1-\nicefrac1{\gamma}}\to \infty$.
If $1<\gamma\leq 2$, we use the lower bound proved in Lemma~\ref{lem:calcul_proba_E1}$(ib)$ for $\delta_n n \gg \omega_n \gg n^{\nicefrac1{\gamma}}$, and get
\[
\frac{\mathbb P_{\bs X}( \mathcal E_{3,2}^{(J_n)})}{\mathbb P_{\bs X}(\mathcal E_1)}
\leq O_\PP\big(n^{\nicefrac1{\gamma}} \big) \Big(\frac{X_n^{_{(2)}}}{X_n^{_{(1)}}}\Big)^{\delta_n n} 
\big(X_n^{_{(1)}}\big)^{-\omega_n}
\leq  O_\PP\big( n^{\nicefrac1{\gamma}}\big)  \mathtt e^{-\Theta_\PP(\delta_n n^{1-\nicefrac1{\gamma}})+
 \Theta_\PP(\omega_n n^{-\nicefrac1{\gamma}})}
= o(1).
\]
It remains to investigate the other summands, corresponding to~$j\not= J_n$. The same argument as above, with $X_j$ playing the role of $X_n^{_{(1)}}$
and $X_n^{_{(1)}}$ playing the role of $X_n^{_{(2)}}$, yields
\begin{align*}
%\sum_{\heap{j=1}{j\not=J_n}}^n 
\mathbb P_{\bs X}\big( \mathcal E_{3,2}^{(j)}\big)
&\leq \cst. \big(X_n^{_{(1)}}\big)^{(\rho_n-\nu_n)n}
\sum_{\varepsilon n<k<c_n n} k^{-\beta} \Big(\frac{X_n^{_{(2)}}}{X_n^{_{(1)}}}\Big)^k
\leq O_\PP\big( n^{\nicefrac1{\gamma}-\beta}\big)  \big(X_n^{_{(1)}}\big)^{(\rho_n-\nu_n)n}  \Big(\frac{X_n^{_{(2)}}}{X_n^{_{(1)}}}\Big)^{\varepsilon n} .
\end{align*}
In the case $\gamma >2$ we can use Lemma~\ref{lem:calcul_proba_E1}$(ia)$ and Lemma~\ref{lem:extremal_properties_X} again and get 
\[
\frac{\sum_{j\not= J_n}\mathbb P_{\bs X}( \mathcal E_{3,2}^{(j)})}{\mathbb P_{\bs X}(\mathcal E_1)}
\leq O_\PP\big( n^{1+\nicefrac1{\gamma}}\big)  \Big(\frac{X_n^{_{(2)}}}{X_n^{_{(1)}}}\Big)^{\varepsilon n}
\leq O_\PP\big( n^{1+\nicefrac1{\gamma}}\big)  \mathtt e^{-\Theta_\PP(\varepsilon n^{1-\nicefrac1{\gamma}})}
= o(1).
\]
If $1<\gamma\leq 2$, we use again Lemma~\ref{lem:calcul_proba_E1}$(ib)$ with $\delta_n n \gg \omega_n \gg n^{\nicefrac1{\gamma}}$, and get
\[
\frac{\sum_{j\not= J_n}\mathbb P_{\bs X}( \mathcal E_{3,2}^{(j)})}{\mathbb P_{\bs X}(\mathcal E_1)}
\leq  O_\PP\big(n^{1+\nicefrac1{\gamma}}\big) \big(X_n^{_{(1)}}\big)^{-\omega_n} \Big(\frac{X_n^{_{(2)}}}{X_n^{_{(1)}}}\Big)^{\varepsilon n}
\leq O_\PP\big(n^{1+\nicefrac1{\gamma}}\big) \mathtt e^{-\Theta_\PP( \varepsilon n^{1-\nicefrac1{\gamma}})+ \Theta_\PP( \omega_n n^{-\nicefrac1{\gamma}})}
= o(1),
\]
as required to prove the claim.
\end{proof}

\begin{proof}[Proof of Theorem~\ref{th:condensation}]
We have proved through Lemmas~\ref{lem:E2<<E1},~\ref{lem:E4<<E1} and~\ref{lem:E3<<E1} 
that, if $\beta+\gamma \geq 2$, we have
$\mathbb{P}_{\bs X}(S_n = m) = (1+o(1))\, \mathbb{P}_{\bs X}(\mathcal{E}_1),$ as $n\uparrow\infty$ and
$\nicefrac{m}{n}\to \rho>\rho^{\star}$. This proves Theorem~\ref{th:condensation}.
\end{proof}

\section{Intermediate symmetry-breaking and the Gamma law}\label{sec:gammas}

\begin{proof}[Proof of Theorem~\ref{th:I_n}]\ \\[-3mm]

{\bf $\bs{(i)}$ The case $\bs{\gamma>1}$.}
We have shown that \smash{$\mathbb{P}_{\bs X}(\mathcal{E}_{1, J_n}^* \st S_n = m) \to 1$} when $n\uparrow\infty$.
Thus, with high probability, the condensate is located at index $J_n$ and its rank is by definition one.

{\bf $\bs{(ii)}$ The case $\bs{\gamma<1}$.}
Let $a, b>0$. Then, by Lemma~\ref{lem:calcul_proba_E1i*},
\begin{align*}
\mathbb{P}_{\bs X}\Big(\big( & \sfrac{K_n}{n^{1-\gamma}}\big)^{\nicefrac1{\gamma}} \in [a,b] \mbox{ and } S_n = m\Big)
= \mathbb{P}_{\bs X}\big(a^{\gamma}n^{1-\gamma} \leq K_n \leq b^{\gamma}n^{1-\gamma}\mbox{ and } 
S_n = m\big)\\
& =  (1+o(1)) \,  \sum_{\heap{i \,\text{such that}}{X_n^{(\lceil a^{\gamma} n^{1-\gamma} \rceil)}\leq X_i \leq X_n^{(\lfloor b^{\gamma} n^{1-\gamma} \rfloor)}} }
\mathbb{P}_{\bs X} (\mathcal{E}_{1,i}^*) \\
& \approx  (1+o(1)){\alpha_2(\rho-\rho^\star)^{-\beta}n^{-\beta}} \, 
\sum_{i=a^{\gamma} n^{1-\gamma}}^{b^{\gamma} n^{1-\gamma}} 
\frac{\big(X_n^{_{(i)}}\big)^{(\rho_n-\nu_n\mp \delta_n)n}}{\Phi\big(X_n^{_{(i)}}\big)} \\
&\approx  (1+o(1)){\alpha_2(\rho-\rho^\star)^{-\beta}n^{-\beta}}\, 
\int_{a^{\gamma} n^{1-\gamma}}^{b^{\gamma} n^{1-\gamma}+1} 
\frac{\big(X_n^{(\lfloor x\rfloor)}\big)^{(\rho_n-\nu_n \mp \delta_n)n}}{\Phi\big(X_n^{(\lfloor x\rfloor)}\big)} \, dx \\
&\approx   (1+o(1)){\alpha_2(\rho-\rho^\star)^{-\beta}n^{1-\beta-\gamma}} \, 
\int_{a}^{b+o(1)} 
\frac{\big(X_n^{(\lfloor y^{\gamma} n^{1-\gamma}\rfloor)}\big)^{(\rho_n-\nu_n \mp \delta_n)n}}{\Phi\big(X_n^{(\lfloor y^{\gamma} n^{1-\gamma}\rfloor)}\big)} \, \gamma y^{\gamma-1} \, dy.
\end{align*}

%In order to apply the law of large numbers for triangular arrays (see, for example,~\cite[page~105]{GK}) to the random variables $Y_{n,i}:=n^{\gamma-1} \indi\{X_i \geq 1 - \nicefrac{x}{n}\}$ we note that
%$\mathtt E Y_{n,i}\sim\nicefrac{x^{\gamma}\alpha_1}n$ and hence
%$n \mathtt P(|Y_{n,i}-\mathtt E Y_{n,i}| >1) \to 0$
%and $n \mathtt E[(Y_{n,i}-\mathtt E Y_{n,i})^2] \to 0$. 
Note that, in view of Assumption~\eqref{eq:rvmu}, 
$\mathbb E\big[n^{\gamma-1}\big|\{i \colon X_i\geq 1-\nicefrac{x}{n}\}\big|\big] \sim \alpha_1 x^{\gamma}$
and $\Var \big[n^{\gamma-1}\big|\{i \colon X_i\geq 1-\nicefrac{x}{n}\}\big|\big] =o(1)$ when $n\uparrow\infty$.
Hence, by Chebyshev's inequality, for all $x \geq 0$, in $\mathtt P$-probability,
\[n^{\gamma-1} \Big|\big\{i \colon X_i\geq 1-\nicefrac{x}{n}\big\}\Big|
\to \alpha_1 x^{\gamma}.\]
Thus, in $\mathtt P$-probability,
$X_n^{(\lfloor y^\gamma n^{1-\gamma} \rfloor)} \sim 1-\frac{y}{n{\alpha_1}^{\nicefrac{1}{\gamma}}}$  as $n\uparrow\infty$,
which implies
\begin{align*}
\mathbb{P}_{\bs X}\Big(\big(\sfrac{K_n}{n^{1-\gamma}}\big)^{\nicefrac1{\gamma}} \in [a,b]\mbox{ and } S_n = m\Big)
&\approx (1+o(1)) \,  {\alpha_2(\rho-\rho^\star)^{-\beta}n^{1-\beta-\gamma}} \,\int_a^{b+o(1)}
\frac{\mathtt e^{-(\rho_n-\nu_n\mp \delta_n) y\alpha_1^{-1/\gamma}}}{\Phi\big(1-y\alpha_1^{-\nicefrac1\gamma} n^{-1}\big)} \, \gamma y^{\gamma-1} \, dy\\
&= (1+o(1)) \,  {\alpha_2 \gamma (\rho-\rho^\star)^{-\beta}n^{1-\beta-\gamma}} \,\int_a^{b}
\exp\big(-(\rho-\rho^{\star})\alpha_1^{-\nicefrac1\gamma} y \big) \, y^{\gamma-1} \, dy.
\end{align*}
Now recall from Lemma~\ref{lem:calcul_proba_E1}\,$(ii)$ that 
$\mathbb{P}_{\bs X}( S_n=m) =  (1+o(1)) \, \alpha_1\alpha_2 (\rho-\rho^\star)^{-\beta-\gamma} \Gamma(\gamma+1) n^{1-\beta-\gamma},$ to obtain
$$\mathbb{P}_{\bs X}\Big(\big(\sfrac{K_n}{n^{1-\gamma}}\big)^{\nicefrac1{\gamma}} \in [a,b]\, \Big| \, S_n = m\Big) = (1+o(1)) \,  \frac{(\rho-\rho^\star)^{\gamma}}{ \alpha_1 \Gamma(\gamma)} \,\int_a^{b}
\exp\big(-(\rho-\rho^{\star})\alpha_1^{-\nicefrac1\gamma} y \big) \,  y^{\gamma-1} \, dy,$$
which concludes the proof of Theorem~\ref{th:I_n}.
\end{proof}

\begin{proof}[Proof of Theorem~\ref{th:gamma}]
Fix $u>0$, $\Delta>0$ and calculate
\[\mathbb{P}_{\bs X}\big(n(1-F_n) \leq u \text{ and } S_n=m\big)
= \mathbb{P}_{\bs X}\big(F_n\geq 1-\nicefrac{u}{n} \text{ and } S_n=m\big)
=(1+o(1))\sum_{\heap{i \, \text{such that}}{ X_i\geq 1-\nicefrac{u}{n}}} 
\mathbb{P}_{\bs X}\big(\mathcal{E}_{1, i}^*\big),\]
since we have shown that $\mathbb{P}_{\bs X}(\bigcup_{i=1}^n \mathcal{E}_{1, i}^* \st S_n = m) \to 1$ when $n\uparrow\infty$.
Thus,
\[\mathbb{P}_{\bs X}\big(n(1-F_n) \leq u \text{ and } S_n=m\big)
\approx \sum_{k=0}^{\frac u\Delta-1} 
\sum_{\heap{i \, \text{such that}}{X_\in[1-\Delta\frac{k+1}{n}, 1-\Delta\frac{k}{n})}}
\alpha_2 (\rho-\rho^{\star})^{-\beta} n^{-\beta} \frac{X_i^{(\rho_n-\nu_n\mp\delta_n)n}}{\Phi(X_i)},
\]
in view of Lemma~\ref{lem:calcul_proba_E1i*}.
It implies
\begin{align*}
\mathbb{P}_{\bs X}\big(n(1-F_n) \leq u \text{ and } S_n=m \big)
&\geq
\sum_{k=0}^{\frac u\Delta-1} 
\sum_{\heap{i \, \text{such that}}{X_i\in[1-\Delta\frac{k+1}{n}, 1-\Delta\frac{k}{n})}}
\alpha_2 (\rho-\rho^{\star})^{-\beta} n^{-\beta} \left(1-\frac{\Delta(k+1)}{n}\right)^{(\rho_n-\nu_n+\delta_n)n} (1+o(1))\\
&\geq
\sum_{k=0}^{\frac u\Delta-1} 
N_k(n)\,
\alpha_2 (\rho-\rho^{\star})^{-\beta} n^{-\beta} \left(1-\frac{\Delta(k+1)}{n}\right)^{(\rho_n-\nu_n+\delta_n)n} (1+o(1)),
\end{align*}
where $N_k(n) = \left|\left\{i \colon  X_i\in[1-\Delta\frac{k+1}{n}, 1-\Delta\frac{k}{n})\right\}\right|$.
Estimating the expectation and variance of $N_k(n)$ and applying Chebyshev's inequality gives, in $\mathtt P$-probability, 
$$n^{\gamma-1} N_k(n) \to \alpha_1 \Delta^\gamma\big( (k+1)^\gamma-k^\gamma\big),$$ 
if $n\uparrow\infty$ and $0\leq k< \frac{u}{\Delta}$.
Thus,
\begin{align*}
\mathbb{P}_{\bs X}\big(n(1-F_n) \leq u \text{ and } S_n=m\big)
&\geq \alpha_1\alpha_2 (\rho-\rho^{\star})^{-\beta} n^{1-\beta-\gamma}
\Delta^\gamma
\sum_{k=0}^{\frac u\Delta-1} 
 \big( (k+1)^\gamma-k^\gamma\big)
\mathtt e^{-(\rho_n-\nu_n+\delta_n)\Delta(k+1)} (1+o(1))\\
&\geq \alpha_1\alpha_2\gamma (\rho-\rho^{\star})^{-\beta} n^{1-\beta-\gamma} \mathtt e^{-(\rho-\rho^\star)\Delta}
\int_{0}^{u} 
x^{\gamma-1}
\mathtt{e}^{-(\rho_n-\nu_n+\delta_n)x} dx\ (1+o(1)),
\end{align*}
because the function $x\mapsto x^{\gamma-1}
\mathtt{e}^{-(\rho_n-\nu_n+\delta_n)x}$ is decreasing on $(0, \infty)$.
Recall that, as $\gamma < 1$, we have
$\mathbb P_{\bs X}(S_n=m) = \alpha_1\alpha_2 (\rho-\rho^{\star})^{-\beta-\gamma} \Gamma(\gamma+1) n^{1-\beta-\gamma} (1+o(1)).$
Together, this implies
\[\liminf_{n\to\infty}
\mathbb{P}_{\bs X}\big(n(1-F_n) \leq u \big| S_n=m\big)
\geq \frac{(\rho-\rho^\star)^{\gamma}}{\Gamma(\gamma)}
\mathtt e^{-(\rho-\rho^\star)\Delta}
{\int_0^{u} x^{\gamma-1} \mathtt{e}^{-(\rho-\rho^{\star})x} dx},
\]
and letting $\Delta\downarrow 0$  concludes the proof.
\end{proof}

\section{Fluctuations of the condensate in the weak disorder case}\label{sec:fluctuations}
In this section, we prove Theorem~\ref{thm:fluctuations}. It follows by combining
Proposition~\ref{prop:TCL}  with the following result.

\begin{prop}
Let $2-\beta<\gamma\leq 1$ and $\rho>\rho^\star$. 
For all $u\in \mathbb R$ there exists $u_n\downarrow0$ such that, with high $\mathtt P$-probability as $n\uparrow\infty$, we have
\[\mathbb P_{\bs X}\Big(\frac{Q_n^{_{(1)}} - m + n \nu_n}{n^{\kappa}}\leq u\ \Big|\ S_n = m\Big)
\approx (1+o(1))\, \mathbb P_{\bs X}\left(\frac{\nu_n n - \sum_{i=1}^n  Q_i}{n^{\kappa}}\leq u \pm u_n\right),\]
where $\kappa=\frac12$, if $\beta+\gamma\geq 3$, and $\kappa=\frac1{\beta+\gamma-1}$ otherwise.
\end{prop}

\begin{proof}
By Lemma~\ref{lem:calcul_proba_E1} we have
$$\mathbb P_{\bs X}\Big(\frac{Q_n^{_{(1)}} - m +n\nu_n}{n^{\kappa}}\leq u\ \Big|\ S_n = m\Big)
\sim \frac{\sum_{i=1}^n \mathbb P_{\bs X}(\mathcal E_{1,i}^* \cap \{Q_n^{(1)} - (\rho_n -\nu_n)n\leq u n^{\kappa}\})}
{\sum_{i=1}^n \mathbb P_{\bs X}(\mathcal E_{1,i}^*)}.$$
The right hand side can be written as
\begin{align}
&\frac
{\displaystyle \sum_{i=1}^n \sum_{\heap{-\delta_n n\leq k-n(\rho_n-\nu_n)}{\leq u n^{\kappa}}} 
\mathbb P_{\bs X} (Q_i = k) \mathbb P_{\bs X}\Big(\sum_{j\neq i}Q_j = m-k\Big)}
{\displaystyle \sum_{i=1}^n \sum_{|k-n(\rho_n-\nu_n)|\leq \delta_n n} 
\mathbb P_{\bs X} (Q_i = k) \mathbb P_{\bs X}\Big(\sum_{j\neq i}Q_j = m-k\Big)}\notag\\
&\phantom{lustiglustig}\approx \frac
{\displaystyle \sum_{i=1}^n  \frac{X_i^{n(\rho_n-\nu_n) \mp \delta_n n}}{\Phi(X_i)}\ 
\mathbb P_{\bs X}\Big(n\nu_n  - u n^{\kappa}\leq \sum_{j\neq i} Q_j\leq n\nu_n +\delta_n n\Big)}
{\displaystyle \sum_{i=1}^n  \frac{X_i^{n(\rho_n-\nu_n) \pm \delta_n n}}{\Phi(X_i)}\ 
\mathbb P_{\bs X}\Big(n\nu_n -\delta_n n\leq \sum_{j\neq i} Q_j \leq n \nu_n + \delta_n n\Big)}\ (1+o(1)). \label{rhsflu}
\end{align}
Note that $\max_{i=1..n} \mathbb P_{\bs X}(Q_i\geq a_n)\to 0$ for any $a_n\uparrow\infty$. Hence we
can find $u_n\downarrow0$ such that, for all $i\in\{1, \ldots, n\}$,
\[\mathbb P_{\bs X}\Big(n\nu_n  - u n^{\kappa}\leq \sum_{j\neq i} Q_j\leq n\nu_n +\delta_n n\Big)
\approx (1+o(1))\, \mathbb P_{\bs X}\Big(n \nu_n -(u\pm u_n)n^\kappa \leq \sum_{j=1}^n Q_j  
\leq n\nu_n+\delta_n n \pm \sfrac12 \delta_n n\Big),\]
where the $o(1)$-term can be chosen independently of $i$.
Therefore, using the choice of $\delta_n$ and a similar bound for the probability in the denominator, we see that
\eqref{rhsflu} is
$$\approx (1+o(1))\,  \frac
{\displaystyle \sum_{i=1}^n  \frac{X_i^{n(\rho_n-\nu_n) \mp \delta_n n}}{\Phi(X_i)}}
{\displaystyle \sum_{i=1}^n  \frac{X_i^{n(\rho_n-\nu_n) \pm \delta_n n}}{\Phi(X_i)}} \,
\mathbb P_{\bs X}\Big(n\nu_n  - (u\pm u_n)n^\kappa\leq \sum_{j=1}^n Q_j\Big).$$
In view of Lemma~\ref{lem:approx_sum}$(ii)$, using that $\gamma<1$, we get that
\begin{align*}
\frac
{\displaystyle \sum_{i=1}^n  \frac{X_i^{n(\rho_n-\nu_n) \mp \delta_n n}}{\Phi(X_i)}}
{\displaystyle \sum_{i=1}^n  \frac{X_i^{n(\rho_n-\nu_n) \pm \delta_n n}}{\Phi(X_i)}}
%&= \frac{\sum_{i=1}^n  X_i^{n(\rho_n-\nu_n) \mp \delta_n n}}{\sum_{i=1}^n  X_i^{n(\rho_n-\nu_n) \pm \delta_n n}}\ (1+o(1))
%= \left(\frac{n(\rho_n-\nu_n) \mp \delta_n n}{n(\rho_n-\nu_n) \pm \delta_n n}\right)^{-\gamma}\ (1+o(1))\\
&\approx \left(1\mp \frac{2\delta_n}{\rho_n-\nu_n}\right)^{-\gamma}\ (1+o(1))
= (1+o(1)).
\end{align*}
We have thus proved the statement in the case $\gamma < 1$.
%\vspace{\baselineskip}
%
%If $\gamma=1$, then,
%\begin{align*}
%\frac
%{\displaystyle \sum_{i=1}^n  \frac{X_i^{n(\rho_n-\nu_n) \mp \delta_n n}}{\Phi(X_i)}}
%{\displaystyle \sum_{i=1}^n  \frac{X_i^{n(\rho_n-\nu_n) \pm \delta_n n}}{\Phi(X_i)}}
%&\approx \left(X_n^{_{(1)}}\right)^{\mp \delta_n n}
%= \exp(\mp \Theta_\PP(\delta_n)) \to 1
%\end{align*}
%when $n\uparrow\infty$, because $X_n^{_{(1)}} = 1-\Theta_{\PP}(n^{-1})$ in view of Lemma~\ref{lem:extremal_properties_X} and since $\gamma=1$.
%The statement is thus also proved for the boundary case $\gamma=1$.
\end{proof}

\section{Further comments and open questions}

\subsubsection*{Fluctuations in the presence of strong disorder.}
Our result on quenched fluctuations in the size of the condensate, Theorem~\ref{thm:fluctuations}, 
is restricted to the weak disorder regime $\gamma<1$. We now give some hints how fluctuations could be treated 
in the strong disorder case. We do not provide details since the focus of the paper is on the weak disorder case.
\smallskip
%\vspace{\baselineskip}

In the case $1< \gamma <2$ the assumption $n^{\nicefrac1{\gamma}} = o(\delta_n n)$ made in the proof of Theorem~\ref{th:condensation}, 
and used to prove Lemma~\ref{lem:E3<<E1} makes $\delta_n$ too large to control precisely the fluctuations of the size of the condensate.
We believe that with some extra effort this assumption can be dropped and Theorem~\ref{thm:fluctuations} can be extended verbatim to 
this regime. 
\smallskip

When $\gamma\geq 2$ more significant changes to the statement proof of Theorem~\ref{thm:fluctuations} are needed. It turns out that 
due to the large fluctuations of the fitness values in this regime the random variables $\sum_{i=1}^n Q_i$ in the grand canonical framework are 
not a sufficiently good approximation of the size of the condensate in the canonical framework. A solution to this problem comes from 
renormalising the fitnesses by their maximum. More precisely, for any~$n$, 
let $\bar X_{i,n} = X_i/ {X_n^{_{(1)}}}$, for all $i\in\{1, \ldots, n\}$. Note that 
the renormalised fitnesses $(\bar X_{1,n}, \ldots, \bar X_{n,n})$ are no longer independent random variables, but
$$\lim_{n\to\infty} \sup_{1\leq i\leq n} \bar X_{i,n}/X_i =1, \qquad \mbox{  in $\mathtt P$-probability}. $$
Defining the random variables $\bar Q_{1,n}, \ldots, \bar Q_{n,n}$ by
\[\mathbb P_{\bs X}(\bar Q_{i,n} = k) = \frac{p_k \bar X_{i,n}^k}{\Phi(\bar X_{i,n})}, \qquad\mbox{ for all }k\in\N,\]
it is straightforward to see from Equation~\eqref{eq:stat_dist} that the law of $(\bar Q_{1,n}, \ldots, \bar Q_{n,n})$
conditional on $\sum_{i=1}^n \bar Q_{i,n} = m$ is equal to the law of $(Q_1, \ldots, Q_n)$ under $P_{m,n}$.
Analysing this ensemble would permit to prove that, if $\rho> \rho^{\star}$, we have in quenched distribution,
\[\frac{Q_n^{(1)}-(m-\bar \nu_n n)}{n^{\nicefrac12}} \to W,\]
where $W$ is a normally distributed random variable, and
\[\bar \nu_n := \frac1n \sum_{\heap{i=1}{i\neq J_n}}^n \mathbb E_{\bs X} \bar Q_{i,n},\]
where $J_n\in\{1,\ldots,n\}$ is the index with $X_{J_n}=X_n^{_{(1)}}$.
Note that the size of the condensate is approximated by $m-\bar \nu_n n$ and not by $m-\nu_n n$ as in Theorem~\ref{thm:fluctuations}.
If $\gamma>2$ this makes a difference. Indeed, by a Taylor expansion of the 
function \smash{$x\mapsto \nicefrac{x\Phi'(x)}{\Phi(x)}$}, using that \smash{$X_n^{_{(1)}} = 1 - \Theta_\PP(n^{-\nicefrac1{\gamma}})$, }
one can see that, in \mbox{$\mathtt P$-probability}, the scaled difference $\sqrt{n}\, (\nu_n - \bar\nu_n)$ tends to zero when $\gamma < 2$ 
but does not tend to zero when $\gamma \geq 2$.

\subsubsection*{Behaviour at criticality.}
In the present article, we assume that the density of particles $\rho_n := \nicefrac m n \to \rho > \rho^{\star}$ when $n\uparrow\infty$.
It would be interesting to \emph{zoom into the transition window}, assuming that $\rho_n$ behaves like $\rho_n = \rho^{\star} + \varepsilon_n$ for some $\varepsilon_n\downarrow 0$. How does the phase transition manifest itself at criticality?

\subsubsection*{Strong excess of particles.}
In another direction, it could be of interest to understand how the system behaves when the average number of particles in the 
grand canonical model is no longer of order $\rho^{\star} n$, but of order $\rho n^{\eta}$ where $\eta>1$. 
Under which condition on $\beta, \gamma, \eta$ do we have condensation? Where is the condensation happening? 
What is the size of the condensate?

\appendix

\section{Random variables near their essential supremum}\label{app:X}

This appendix is devoted to asymptotic properties of  a random variable~$X$ with distribution~$\mu$ on~$[0,1]$,
which satisfies~\eqref{eq:rvmu}. We denote by $(X_i)_{i\in \N}$ an i.i.d. sequence of random variables with the same
distribution as~$X$. Let \smash{$(X_n^{_{(1)}}, \ldots, X_n^{_{(n)}})$}  be the order statistics  of $(X_1, \ldots, X_n)$.  
In some results we additionally refer to a continuous function $\Psi \colon [0,1] \to (0, \infty)$ such that $\Psi(1)=1$.
\medskip

This first lemma is a classical result for regularly varying random variables:
\begin{lem}[{see~\cite[Chapter~0.4]{Resnick}}]\label{lem:extremal_properties_X}
We have, in probability as $n\uparrow\infty$, % Here implied constants are random!
\[1-X_n^{_{(1)}} = \Theta_{\mathtt P}\big(n^{-\nicefrac1{\gamma}}\big) 
\quad \text{ and }\quad 
1-\frac{X_n^{_{(2)}}}{X_n^{_{(1)}}} = \Theta_{\mathtt P}\big(n^{-\nicefrac1{\gamma}}\big).\]
\end{lem}

\begin{lem}\label{lem:moments}
As $r\uparrow\infty$, we have $\mathtt E \left(\frac{X^r}{\Psi(X)}\right) \sim \alpha_1 \Gamma(\gamma+1)\, r^{-\gamma}$.
\end{lem}

\begin{proof}
First note that
\[\mathtt E\left[\frac{X^r}{\Psi(X)}\right]
= \mathtt E\left[\frac{X^r}{\Psi(X)}\ \indi\{X> 1-\nicefrac{2\gamma \log r}{r}\}\right]
+ \mathtt E\left[\frac{X^r}{\Psi(X)}\ \indi\{X\leq 1-\nicefrac{2\gamma \log r}{r}\}\right].\]
The second term of the above sum verifies
\[\mathtt E\left[\frac{X^r}{\Psi(X)}\ \indi\{X\leq 1-\nicefrac{2\gamma \log r}{r}\}\right]
\leq r^{-2\gamma} \mathtt E\left[\frac{1}{\Psi(X)}\right] \leq \frac1{p_0} r^{-2\gamma},\]
since $\Psi$ is bounded from below by some $p_0>0$ on $[0,1]$.
The fact that $\Psi$ is continuous in 1 gives that
\[\mathtt E\left[\frac{X^r}{\Psi(X)}\ \indi\{X> 1-\nicefrac{2\gamma \log r}{r}\}\right]
= (1+o(1))\ \mathtt E[X^r\ \indi\{X> 1-\nicefrac{2\gamma \log r}{r}\}].\]
By Fubini's theorem, and in view of Assumption~\eqref{eq:rvmu}, we have
\begin{align*}
\mathtt E[ X^r\ \indi\{X> 1-\nicefrac{2\gamma \log r}{r}\}]
&= \int_0^1 \mathtt P(X^r\ \indi\{X> 1-\nicefrac{2\gamma \log r}{r}\}\geq x) \, dx \\
&= \int_0^{(1-\nicefrac{2\gamma\log r}{r})^r} \mathtt P(X >  1-\nicefrac{2\gamma \log r}{r})\, dx
+  \int_{(1-\nicefrac{2\gamma\log r}{r})^r}^1 \mathtt P(X^r\geq x) \,dx.\\
&= (1+o(1))\ r^{-2\gamma} \mu(1-\nicefrac{2\gamma \log r}{r}, 1)
+  \int_{(1-\nicefrac{2\gamma\log r}{r})^r}^1 \mu(x^{\nicefrac 1 r}, 1) \,dx\\
&= o(r^{-\gamma}) + \int_0^{2\gamma\log r} \mu(1-\nicefrac z r, 1) (1-\nicefrac z r)^{r-1}\, dz,
\end{align*}
by the change of variables $r(1-x^{\nicefrac1r}) = z$, $dx = -(1-\nicefrac z r)^{r-1} dz$.
Thus by Equation~\eqref{eq:rvmu}, we get
\[\mathtt E[ X^r\ \indi\{X> 1-\nicefrac{2\gamma \log r}{r}\}] 
=\alpha_1 r^{-\gamma} \int_0^{2\gamma\log r} z^{\gamma} (1-\nicefrac z r)^{r-1} dz+o(r^{-\gamma})
= (\alpha_1+o(1))\, r^{-\gamma} \int_0^{\infty} z^{\gamma} \mathtt e^{-z} dz,\]
which concludes the proof.
\end{proof}

Note that $\mathtt E X^r \sim \alpha_1 \Gamma(\gamma+1)\, r^{-\gamma}$ as $r\uparrow\infty$,
by choosing $\Psi(x)=1$ for all $x\in[0,1]$.

\begin{lem}\label{lem:approx_sum}\ \\[-5mm]
\begin{enumerate}[(i)]
\item 
% If 1 is replaced by EX^k this is the LLN.
%\item Let $k$ be an integer, then, in $\mathtt P$-probability as $n\uparrow\infty$,
%$\sum_{i=1}^n X_i^k = n(1+o(1))$.
For all $n\geq 1$ and $k\geq 0$, let $$V_k^{_{(n)}} := \frac{\sum_{i=1}^n X_i^k}{(X_n^{_{(1)}})^k}.$$ 
The sequence $(V_k^{_{(n)}})_{k\geq 0}$ is non-increasing for all integer $n$.

Let $(s_n)_{n\geq 1}$ be a sequence of positive reals,
such that $s_n \gg n^{\nicefrac1{\gamma}} \log n$. 
Then,
\[\lim_{n\to\infty} V_{s_n}^{(n)} = 1 \mbox{  in $\mathtt P$-probability.}\]
\item  For all $n\geq 1 $ and for all $k\geq 0$, let $U_0^{_{(n)}} := 0$ and
\[U_k^{_{(n)}}:= \frac{k^{\gamma}}{n} \sum_{i=1}^n \frac{X_i^k}{\Psi(X_i)}.\]
Let $(s_n)_{n\geq 1}$ be a sequence of positive reals, such that $s_n\ll n^{\nicefrac1{\gamma}}$.
Then,
\[\lim_{n\to\infty} U_{s_n}^{(n)} =  \alpha_1\Gamma(\gamma+1) \mbox{  in $\mathtt P$-probability.}\]
\item For all constants $c>0$, the sequence
$(\sum_{i=1}^n X_i^{c n^{\nicefrac1{\gamma}}})_{n\geq 1}$ is tight.
\end{enumerate}
\end{lem}

\begin{proof}
$(i)$ 
Fix $n\geq 1$, then, for all $k\geq 0$,
\[\frac{V_{k}^{(n)}}{V_{k+1}^{(n)}}
= X_n^{(1)} \ \frac{\sum_{i=1}^n X_i^k}{\sum_{i=1}^n X_i^{k+1}}
\geq 1,\]
using that $X_i \leq X_n^{(1)}$ for all $i\in\{1, \ldots, n\}$.
Now, observe that, $\sum_{i=2}^n (X_n^{_{(i)}})^{s_n} \leq n (X_n^{_{(2)}})^{s_n},$
which implies that
\[\frac{\sum_{i=2}^n (X_n^{_{(i)}})^{s_n}}{(X_n^{_{(1)}})^{s_n}}
\leq n \Big(\frac{X_n^{_{(2)}}}{X_n^{_{(1)}}}\Big)^{s_n}
= n \big(1- \Theta_{\mathtt P}\big(n^{-\nicefrac1{\gamma}}\big)\big)^{s_n}=o(1), \]
which concludes the proof of $(i)$.

\vspace{\baselineskip}
$(ii)$ We have, as $n\to\infty$, in view of Lemma~\ref{lem:moments},
\[\mathbb E\left[\frac{s_n^{\gamma}}{n} \sum_{i=1}^n \frac{X_i^{s_n}}{\Psi(X_i)}\right]
=  s_n^{\gamma} \mathbb E\left[\frac{X_i^{s_n}}{\Psi(X_i)}\right]\sim \alpha_1\Gamma(1+\gamma).\]
Moreover, applying Lemma~\ref{lem:moments} again and denoting by $p_0$  the positive lower bound of $\Psi$ on $[0,1]$,
\[\Var \left[\frac{s_n^{\gamma}}{n} \sum_{i=1}^n \frac{X_i^{s_n}}{\Psi(X_i)}\right]
= \frac{s_n^{2\gamma}}{n} \Var\left[\frac{X_i^{s_n}}{\Psi(X_i)}\right]
\leq \frac{s_n^{2\gamma}}{p_0^2 n}\ \mathbb E X_i^{2s_n} 
=o(1).\]
The statement now follows by Chebyshev's inequality.

\vspace{\baselineskip}
$(iii)$
Note that, as $n\uparrow\infty$, in view of Lemma~\ref{lem:moments},
\[\mathbb E\left[\sum_{i=1}^n X_i^{cn^{\nicefrac1{\gamma}}}\right]
= n \mathbb E \big[X^{n^{\nicefrac1{\gamma}}} \big] = c^{-\gamma}\alpha_1\Gamma(1+\gamma).\]
Similarly,
$\Var\left(\sum_{i=1}^n X_i^{c n^{\nicefrac1{\gamma}}}\right)
= n \Var \big(X^{c n^{\nicefrac1{\gamma}}}\big)  = O(1),$
which implies the result by Chebyshev's inequality.
\end{proof}

\end{document}